\documentclass[a4paper]{article}

\usepackage[all]{xy}\usepackage[latin1]{inputenc}        
\usepackage[dvips]{graphics,graphicx}
\usepackage{amsfonts,amssymb,amsmath,color,mathrsfs, amstext,bm}
\usepackage{amsbsy, amsopn, amscd, amsxtra, amsthm,authblk, enumerate}
\usepackage{upref}
\usepackage[colorlinks,
            linkcolor=red,
            anchorcolor=red,
            citecolor=red
            ]{hyperref}

\usepackage{geometry}
\usepackage[displaymath]{lineno}
\usepackage{float}

\usepackage{yhmath}

\usepackage{enumitem}

\usepackage[normalem]{ulem}

\usepackage{xcolor}

\numberwithin{equation}{section}

\usepackage{graphicx}
\usepackage{subcaption}

\usepackage{algorithm}
\usepackage{algpseudocode}
\usepackage{caption}

\newtheorem{theorem}{Theorem}
\newtheorem{lemma}{Lemma}
\newtheorem{assumption}{Assumption}

\newtheorem{corollary}{Corollary}

\newcommand\keywords[1]{\textbf{Keywords:} #1}

\begin{document}
\title{On Random Batch Methods (RBM) for interacting particle systems driven by L\'evy processes}
\author[a]{Jian-Guo Liu\thanks{E-mail:jliu@math.duke.edu}}
\author[b]{Yuliang Wang\thanks{Email:YuliangWang$\_$math@sjtu.edu.cn}}
\affil[a]{Department of Mathematics, Department of Physics, Duke University, Durham, NC 27708, USA.}
\affil[b]{School of Mathematical Sciences, Institute of Natural Sciences, Shanghai Jiao Tong University, Shanghai, 200240, P.R.China.}

\date{}
\maketitle

\begin{abstract}
In many real-world scenarios, the underlying random fluctuations are non-Gaussian, particularly in contexts where heavy-tailed data distributions arise. A typical example of such non-Gaussian phenomena calls for L\'evy noise, which accommodates jumps and extreme variations. We propose the Random Batch Method for interacting particle systems driven by L\'evy noises (RBM-L\'evy), which can be viewed as an extension of the original RBM algorithm in [Jin et al. J Compt Phys, 2020]. In our RBM-L\'evy algorithm, $N$ particles are randomly grouped into small batches of size $p$, and interactions occur only within each batch for a short time. Then one reshuffles the particles and continues to repeat this shuffle-and-interact process. In other words, by replacing the strong interacting force by the weak interacting force, RBM-L\'evy dramatically reduces the computational cost from $O(N^2)$ to $O(pN)$ per time step. Meanwhile, the resulting dynamics converges to the original interacting particle system, even at the appearance of the L\'evy jump. We rigorously prove this convergence in Wasserstein distance, assuming either a finite or infinite second moment of the L\'evy measure. Some numerical examples are given to verify our convergence rate.
\end{abstract}


\section{Introduction}

Interacting particle systems can effectively model various systems in both natural and social sciences, such as molecules in fluids \cite{frenkel2023understanding}, plasma \cite{birdsall2018plasma}, swarming \cite{carlen2013kinetic,carrillo2017particle,degond2017coagulation,vicsek1995novel}, chemotaxis \cite{bertozzi2012characterization,horstmann20031970}, flocking \cite{albi2013binary,cucker2007emergent,ha2009simple}, synchronization \cite{choi2011complete,ha2009simple}, consensus \cite{motsch2014heterophilious}, random vortex model \cite{majda2002vorticity}, and beyond. In the interacting particle system, each individual is called a particle, and they interact with each other according to some interacting force, which can be derived from existing physical or economical theories such as the Newton's second law. Traditionally, people add white noise with continuous trajectories to describe the uncertainty of particle movement. However, in many applications, the random fluctuation of is usually non-Gaussian, and people usually use L\'evy-type noises that allow jumps to model these phenomena \cite{ditlevsen1999observation,shkolnikov2011competing,huang2021microscopic,wang2024small,blanchet2024limit}. In this paper, we focus on efficient simulation and its rigorous convergence analysis of interacting particle systems driven by (pure jump) L\'evy processes. For the pairwise interacting case, direct evaluation of the interaction requires complexity of $O(N^2)$ per time step. Inspired by the Random Batch Method (RBM) (for the Gaussian noise case) proposed in \cite{jin2020random}, we extend this idea to the case where the noise in the system has jumps. In detail, the random shuffle-interact idea of RBM can be described as follows (also see in Algorithm \ref{al:RBM1} below): in each ($k$-th) time interval, we randomly divide $\{1,\dots,N\}$ into $\frac{N}{p}$ small batches, each with batch size $p$ ($p \ll N$, often $p=2$), denoted by $\mathcal{C}_{k,q}$ ($q=1,2, \dots, \frac{N}{p}$), and particles interact among each other within each batch. Then, the particles are reshuffled, and the process continues. Consequently, the computational cost is reduced to $O(pN)$ per time step while one may still have convergence, even in the L\'evy noise case. In this paper, we provide the first rigorous proof for the convergence guarantee of RBM with L\'evy noise and bounded interacting kernel with an explicit convergence rate. The convergence analysis is valid for both finite and infinite second-moment cases. In particular, our result is valid for the rotational invariant $\alpha$-stable case, which is applicable in various practical tasks.

Let us first introduce the basic settings of this paper.
Consider the following first-order pairwise interacting particle system driven by L\'evy processes:
\begin{equation}\label{eq:ips}
    X^{i}(t) = X^{i}(0) - \int_0^t \nabla V(X^i(s)) ds + \int_0^t \frac{1}{N-1} \sum_{j\neq i}K(X^j(s) - X^i(s)) ds +   L^{i}(t),\quad 1 \leq i\leq N,
\end{equation}
where $X^i \in \mathbb{R}^d$, $V:\mathbb{R}^d \rightarrow \mathbb{R}^d$, $K:\mathbb{R}^d \rightarrow \mathbb{R}^d$,  and $L^i(t)$ are $N$ independent L\'evy processes with L\'evy measure $\nu$ satisfying $\nu(\{0\}) = 0$ and
\begin{equation}
    \int_{\mathbb{R}^d} 1 \wedge |z|^2 \nu(dz) < \infty.
\end{equation}
Moreover, the L\'evy-It\^o decomposition gives
\begin{equation}
    L(t) = P.V. \int_0^t \int_{|z|<1} z\tilde{N}(ds,dz) + \int_0^t \int_{|z| \geq 1} z N(ds, dz),
\end{equation}
where $N(dt,dz)$ is the Poisson measure and $\tilde{N}(dt,dz):= N(dt,dz) - \mathbb{E}N(dt,dz) = N(dt,dz) - \nu(dz)dt$ is the compensated Poisson measure. Also, P.V. denotes the Cauchy principle value, and in the followings we will omit this notation for simplicity. We would refer the readers to \cite{applebaum2009levy} for more detailed introduction for the L\'evy process. We also provide some related basic knowledge in Appendix \ref{app:levy} of this paper.

As we would see in Algorithm \ref{al:RBM1}, the Random Batch Method replaces the weak interaction force $\frac{1}{N-1}K(X^j - X^i)$ between all pairs $(X^j,X^i)$ with some strong interaction force $\frac{1}{p-1}K(X^j - X^i)$ within pairs selected in a small batch with size $p$. In other words, the main modification in the RBM is replacing  $\frac{1}{N-1}\sum_{j\neq i}K\left(X^j - X^i \right)$ (the total force acting on $X^i$) with $\frac{1}{p-1}\sum_{j\in \mathcal{C},j\neq i}K\left(X^j - X^i \right)$, where $\mathcal{C} \subset \{1,\dots,N\}$ is a random batch of size $p$. 
Consequently, this replacement dramatically reduces the computational cost from $O(N^2)$ to $O(pN)$ per time step, and provides an unbiased approximation for the force / velocity field of the original interacting particle system. 
This means for any fixed deterministic sequence $(x^i)_{i=1}^N$, 
\begin{equation}\label{eq:unbiased}
\mathbb{E}\big[\frac{1}{p-1}\sum_{j\in \mathcal{C},j\neq i}K\left(x^j - x^i \right) \mid i \in \mathcal{C}\big] = \frac{1}{N-1}\sum_{j\neq i}K\left(x^j -x^i \right),
\end{equation}
where the randomness is from the random batch $\mathcal{C}$. The above consistency property is quite intuitive since conditioning on $i \in \mathcal{C}$, the probability that another index $j$ is chosen into $\mathcal{C}$ is $\tfrac{p-1}{N-1}$. This consistency property \eqref{eq:unbiased} of the random batch indicates that the modified dynamics in RBM  should be close to the original interacting particle system, and in fact it plays a significant role in our convergence analysis. 

In detail, the Random Batch Method for \eqref{eq:ips} acts as follows:
Let $\kappa > 0$ be the constant time step, $p$ be the batch size (assume that $p$ divides $N$), and denote $t_m := m\kappa$ for $m = 0,1,2,\dots$. For each $m$, divide $\{1,\dots, N \}$ into $p$ batches $\mathcal{C}_{m,1},\dots,\mathcal{C}_{m,N/p}$ randomly, for each $i$ and $q$ such that $i \in \mathcal{C}_{m,q}$, update the $\tilde{X}^i$ by solving the following SDE for $t \in [t_m,t_{m+1})$:
\begin{equation}\label{eq:rbm}
    \tilde{X}^i(t) = \tilde{X}^i(t_m) - \int_{t_m}^t \nabla V(\tilde{X}^i(s)) ds + \int_{t_m}^t \frac{1}{p-1}\sum_{j\neq i, j \in \mathcal{C}_{m,q}} K(\tilde{X}^j(s) - \tilde{X}^i(s)) ds + ( L^i(t) - L^i(t_m)).
\end{equation}
Then one reshuffles the particles and continues to repeat this shuffle-and-interact procedure. The detailed algorithm is demonstrated in Algorithm \ref{al:RBM1} below.

\begin{algorithm}[H]
\caption{RBM-L\'evy}\label{al:RBM1}
\begin{algorithmic}[1] 
\For{$m \in \{0, \ldots, \frac{T}{\kappa}-1\}$}
    \State Divide $\{1, 2, \ldots, N\}$ into $\tfrac{N}{p}$ batches randomly.
    \For{each batch $\mathcal{C}_{m,q}$ ($q = 1,2,\dots, \tfrac{N}{p}$)} 
        \State Update $X^i$ (where $i \in \mathcal{C}_{m,q}$) by solving the following SDE with $t \in [t_{m}, t_{m+1}]$:
        \begin{equation*}
        dX^i = -\nabla V(X^i)dt + \frac{1}{p-1} \sum_{\substack{j \in \mathcal{C}_{m,q},  j \neq i}} K(X^j - X^i)dt +  dL^i.
        \end{equation*}
    \EndFor
\EndFor
\end{algorithmic}
\end{algorithm}

Originally proposed in \cite{jin2020random}, the Random Batch Method for systems with continuous noises has been applied to design various practical algorithms. For instance, in \cite{jin2021random, liang2022superscalability, liang2022improved, gan2024random}, the authors combined the RBM and the Ewald algorithm and proposed the so-called Random Batch Ewald (RBE) method for the molecular dynamics. Other recent extensions of RBM include its application to sampling \cite{li2020random}, interacting quantum particle systems \cite{golse2019random, jin2020randomquantum}, second-order interacting particle systems \cite{jin2022random}, multi-species stochastic interacting
particle systems \cite{jin2021convergence,daus2022random}, the Poisson-Nernst-Planck (PNP) equation \cite{li2022some}, the Landau equation \cite{carrillo2022random} and Landau-type equations \cite{du2024collision}, to name a few.  On the other hand, instead of the Brownian motion, the usage of L\'evy noise is of great significance in various dynamical systems. In fact, in many real-world scenarios, the underlying random fluctuations are non-Gaussian, particularly in contexts where heavy-tailed data distributions arise. Typical examples include the COVID-19 deaths \cite{cohen2022covid}, hurricane-related losses \cite{de2024prediction}, etc. A typical example of such non-Gaussian phenomena calls for L\'evy noise, which accommodates jumps and extreme variations. For instance, recently the authors in \cite{blanchet2024limit} used the L\'evy noise to study the fluctuation of the stochastic gradient descent (SGD) when the data (random mini-batch) has infinite variance. Also, the authors of \cite{wang2024small} studied the small-mass limit of interacting particle systems driven by L\'evy processes. Moreover, the mean-field theory for interacting particle systems driven by noises with jumps has been investigated \cite{meleard1996asymptotic,liang2021exponential}, and other properties such as the Smoluchowski-Kramers limit for SDEs driven by L\'evy noises has been studied \cite{zhang2008smoluchowski,wang2023homogenization}. However, the efficient simulation of the interacting particle systems driven by L\'evy process remains open. The main contribution of this paper is addressing this problem by applying the random batch idea to simulate and analyze this dynamics with L\'evy jumps, and providing a rigorous convergence guarantee as the time step $\kappa$ tends to zero.

At first glance, after only a few iterations, the resulted dynamics would differ a lot from the original dynamics, due to the difference between the weak interaction and the strong interaction mentioned above. However, for $\kappa \rightarrow 0$, the RBM dynamics would be close to the original one due to the law of large number in time. In detail, we prove in Theorem \ref{thm:nosecondmoment}, \ref{thm:secondmoment} that whether or not the L\'evy noise has finite second moment,  RBM always converges to the interacting particle system uniformly in time under Wasserstein distances (to be more specific, under Wasserstein-2 distance when the second moment is finite, and under Wasserstein-1 distance when the second moment is infinite).

Part of our techniques for the convergence analysis are similar to existing ones (for instance, in \cite{jin2020random,cai2024convergence}). However, due to the discontinuous nature of the L\'evy noise, the associated analysis is not very trivial. For instance, one significant challenge we overcome during the derivation in this paper is:
\begin{itemize}
    \item We address the possible singularity problem in the case when the L\'evy process does not have finite second moment. This problem shares the same nature of the singularity of $x \mapsto \frac{d^2}{dx^2}|x|^\alpha$ at $x = 0$ for $\alpha \in (1,2)$. We solve this issue by introducing a regularized ``norm" $\| \cdot \|_\epsilon = \sqrt{|\cdot|^2 + \epsilon^2}$ and choosing suitable value of $\epsilon$.
\end{itemize}
Notably, our techniques can be naturally extended to analyze other systems, such as the kinetic interacting particle systems with L\'evy noise, and systems with unequal particle masses. However, it is still unclear whether we could establish convergence for more complicated cases, for instance, the space-inhomogeneous noise, and the singular interacting kernels. We leave these as possible future work.

The rest of the paper is organized as follows. In Section \ref{sec:convergence}, we prove the convergence of RBM-L\'evy to the original interacting particle system driven by the L\'evy processes. We consider the L\'evy noise with finite second moments in Section \ref{sec:secondmoment}, and with infinite second moments in Section \ref{sec:no_second}. Some lemmas used in the main proofs are given in Section \ref{sec:lemma}. In Section \ref{sec:num}, we run some numerical tests to evaluate the RBM-L\'evy algorithm and verify our theory in Section \ref{sec:convergence}. 

\section{Convergence anlyasis for Random Batch Methods with L\'evy noise}\label{sec:convergence}
Recall the RBM for interacting particle system with L\'evy noise defined in \eqref{eq:rbm}. In this section, under suitable assumptions for the L\'evy measure $\nu$ as well as for the forces $V(\cdot)$, $K(\cdot,\cdot)$, we prove its convergence to the interacting particle system \eqref{eq:ips} as the time step $\kappa \rightarrow 0$. 

Denote $\rho_t^{(1)}$, $\tilde{\rho}_t^{(1)}$ the first marginal distributions of \eqref{eq:ips}, \eqref{eq:rbm}, respectively, we aim to estimate the Wasserstein-$p$ distance (for $p = 1,2$)
\begin{equation}
    W_p(\rho_t^{(1)}, \tilde{\rho}_t^{(1)}).
\end{equation}
Note that for any two probability measure $\mu$, $\nu$, $W_p(\mu,\nu)$ is defined by
\begin{equation*}
    W_{p}(\mu, \nu):=\left(\inf _{\gamma \in \Pi(\mu, v)} \int_{\mathbb{R}^{d} \times \mathbb{R}^{d}}|x-y|^{p} d \gamma\right)^{1 / p}.
\end{equation*}
In our results below, we only consider $p=1,2$ for simplicity: in Section \ref{sec:secondmoment}, we assume that the L\'evy measure $\nu$ has finite second moment, and show that the error $W_2(\rho_t^{(1)}, \tilde{\rho}_t^{(1)})$ is at most $O(\sqrt{\kappa})$ uniformly in time; in Sectiom \ref{sec:no_second}, we assume that the L\'evy measure $\nu$ only has finite first moment, and show that the error $W_1(\rho_t^{(1)}, \tilde{\rho}_t^{(1)})$ is also at most $O(\sqrt{\kappa})$ uniformly in time.

\subsection{Systems driven by L\'evy processes with finite second moments}\label{sec:secondmoment}
We first consider the case where the Levy process has finite second moment. In detail, we assume:
\begin{assumption}[finite second moment]\label{ass1}
The L\'evy measure $\nu$ satisfies
\begin{equation}
    \int_{|z|\geq 1} |z|^2 \nu(dz) < \infty.
\end{equation}
\end{assumption}
Moreover, we need the following assumptions for $V$, $K$ in the drift and the initial states.
\begin{assumption}\label{ass2}
Assume that $K$ is bounded by $M_K$ and $L_K$-Lipschitz, $V$ is $\lambda_V$-strongly convex (i.e. function $x \mapsto V(x) - \frac{\lambda_V}{2}|x|^2$ is convex), $\nabla V$ is $L_V$-Lipschitz.
\end{assumption}

\begin{assumption}\label{ass:init1}
The initial states satisfy
\begin{equation}
    X^i(0) = \tilde{X}^i(0), \quad \sup_i \mathbb{E}|X^i(0)|^2 <\infty,\quad \sup_i \mathbb{E}|V(X^i(0))|^2 <\infty.
\end{equation}
\end{assumption}

\begin{theorem}\label{thm:secondmoment}
Let $\rho_t^{(1)}$, $\tilde{\rho}_t^{(1)}$ the first marginal distributions of \eqref{eq:ips}, \eqref{eq:rbm}, respectively. 
 Under Assumptions \ref{ass1}, \ref{ass2}, \ref{ass:init1}, for small time step $\kappa$, there exists a positive constant $C$ such that
\begin{equation}
   \sup_{t \geq 0} W_2(\rho_t^{(1)}, \tilde{\rho}_t^{(1)}) \leq C\kappa^{\frac{1}{2}}.
\end{equation}
\end{theorem}

\begin{proof}
We consider the synchronous coupling (same L\'evy noise): for $1 \leq i\leq N$,
\begin{equation}\label{eq:ips1}
    X^{i}(t) = X^{i}(0) - \int_0^t \nabla V(X^i(s)) ds + \int_0^t \frac{1}{N-1} \sum_{j\neq i}K(X^j(s) - X^i(s)) ds + L^{i}(t),
\end{equation}

\begin{equation}\label{eq:rbm1}
    \tilde{X}^i(t) = \tilde{X}^i(0) - \int_{0}^t \nabla V(\tilde{X}^i(s)) ds + \int_0^t \frac{1}{p-1}\sum_{j\neq i, j \in \mathcal{C}_{m,q(i)}} K(\tilde{X}^j(s) - \tilde{X}^i(s)) ds +   L^i(t).
\end{equation}
Recall that $q = 1,2,\dots, \frac{N}{p}$ in Algorithm \ref{al:RBM1}, and for each index $i$, there exists a $q$ such that $i \in \mathcal{C}_{m,q}$. In \eqref{eq:rbm1} and what follows, we denote such $q$ by $q(i)$ so that one always has $i \in \mathcal{C}_{m,q(i)}$. 

Denote
\begin{equation}
    Z^i(t) := \tilde{X}^i(t) - X^i(t).
\end{equation}
Then, by exchangeability and Assumption \ref{ass2}, we have (some trivial details omitted)
\begin{equation}
    \frac{1}{2}\mathbb{E}|Z^1(t)|^2 \leq - (\lambda_V - 2L_K) \int_0^t \mathbb{E}|Z^1(s)|^2 ds + \int_0^t \mathbb{E}\left[Z^1(s) \cdot \mathcal{X}_{m,1}(\tilde{X}(s)) \right]ds,
\end{equation}
where
\begin{equation}\label{eq:mathcalXdef}
    \mathcal{X}_{m,i}(\tilde{X}(s)) := \frac{1}{p-1}\sum_{j\neq i, j \in \mathcal{C}_{m,q(i)}} K(\tilde{X}^j(s) - \tilde{X}^i(s)) - \frac{1}{N-1} \sum_{j\neq i}K(\tilde{X}^j(s) - \tilde{X}^i(s)).
\end{equation}
Denote 
\begin{equation}
    R(t) := \mathbb{E}\left[Z^1(t) \cdot \mathcal{X}_{m,1}(\tilde{X}(t)) \right].
\end{equation}
Then
\begin{equation}
\begin{aligned}
    R(t) &= \mathbb{E}\left[Z^1(t_m) \cdot \mathcal{X}_{m,1}(\tilde{X}_{t_m})\right]\\
    &\quad + \mathbb{E}\left[Z^1(t_m) \cdot (\mathcal{X}_{m,1} (\tilde{X}(t)) - \mathcal{X}_{m,1}(\tilde{X}(t_m)))\right]\\
    &\quad+\mathbb{E}\left[(Z^1(t) - Z^1(t_m)) \cdot \mathcal{X}_{m,1}(\tilde{X}(t))\right]\\
    &:=I_1 + I_2 + I_3.
\end{aligned}
\end{equation}
Clearly,
\begin{equation}\label{eq:I10}
    I_1 = 0.
\end{equation}
In fact, at $t_m$, the random batch $\mathcal{C}_{m,q(i)}$ is independent of $\tilde{X}^i(t_m)$ and $X^i(t_m)$, and since
\begin{equation*}
\mathbb{E}\left[\mathcal{X}_{m,i}\left(\tilde{X}\left(t_k\right)\right)\right] = \mathbb{E}\left[\mathbb{E}\left[\mathcal{X}_{m,i}\left(\tilde{X}\left(t_m\right)\right) \mid \tilde{X}^i(t_m),1\leq i \leq N\right]\right]=0
\end{equation*}
due to the consistency of the random batch, we know that $I_1 = 0$.

For the term $I_2$, by Lemma \ref{lmm:Zcon} below, we have
\begin{equation}
    \mathbb{E}|Z^1(t) - Z^1(t_m)|^2 \leq C\kappa^2,\quad t \in [t_m,t_m).
\end{equation}
Moreover, defining te $\mathcal{F}_{t_m}:= \sigma(\tilde{X}^i(s), X^i(s), \mathcal{C}_{k,q(i)},  k \leq m, 1\leq i \leq N, s \leq t_m)$,  we prove in Lemma \ref{lmm:mathcalX}  below that
\begin{equation}
\mathbb{E}\left[Z^1(t_m) \cdot (\mathcal{X}_{m,1} (\tilde{X}(t)) - \mathcal{X}_{m,1}(\tilde{X}(t_m)))\right] \leq C\kappa \left( \mathbb{E}|Z^1(t_m)|^2\right)^{\frac{1}{2}}.
\end{equation}
Consequently, combining with Lemma \ref{lmm:Zcon} and Young's inequality, for any $\eta>0$ to be determined, we have
\begin{equation}\label{eq:withdelta}
    I_2 \leq  C\kappa^2 + C\kappa \sqrt{\mathbb{E}|Z^1(t)|^2} \leq  (C + \frac{C}{4\eta})\kappa^2 + \eta\mathbb{E}|Z^1(t)|^2.
\end{equation}

For the term $I_3$, since $K$ is bounded by $M_K$ Assumption \ref{ass2}, we know that $\mathcal{X}_{m,1}(\tilde{X}(t))$ is bounded by $2M_K$. Therefore, using Lemma \ref{lmm:Zcon} again, we have 
\begin{equation}
    I_3 \leq C\kappa.
\end{equation}

Finally, combining the estimates for $I_1$, $I_2$, $I_3$, and choosing $\delta$ (in \eqref{eq:withdelta}) small enough such that $\eta < \lambda_V- 2L_K$, we have
\begin{equation}
    \mathbb{E}|Z^1(t)|^2 \leq \int_0^t (-C\mathbb{E}|Z^1(s)|^2 + C\kappa).
\end{equation}
Using the Gr\"onwall's inequality, and by definition of the Wasserstein-2 distance, we obtain the desired uniform-in-time error estimate:
\begin{equation}
   \sup_{t \geq 0} W_2(\rho_t^{(1)}, \tilde{\rho}_t^{(1)}) \leq C\kappa^{\frac{1}{2}}.
\end{equation}
Above, the values of the constant $C$ may vary but are all independent of the time $t$ and the particle number $N$.
\end{proof}

\subsection{Systems driven by L\'evy processes with infinite second moments}\label{sec:no_second}
In many applications, the L\'evy measure does not have finite second-order moment. For instance, if we consider instead the rotational invariant $\alpha$-stable L\'evy process with $\alpha \in (1,2)$, the corresponding RBM algorithm would still have convergence. Moreover, the convergence is still non-asymptotic, thanks to the strong convexity of the external potential $V$.

In detail, we replace Assumption \ref{ass1} by the following weaker condition:

\begin{assumption}[finite first moment]\label{ass00}
The L\'evy measure $\nu$ satisfies
\begin{equation}
    \int_{|z|\geq 1} |z| \nu(dz) < \infty.
\end{equation}
\end{assumption}

\begin{theorem}\label{thm:nosecondmoment}
Let $\rho_t^{(1)}$, $\tilde{\rho}_t^{(1)}$ the first marginal distributions of \eqref{eq:ips}, \eqref{eq:rbm}, respectively. 
 Under Assumptions \ref{ass2}, \ref{ass:init1}, \ref{ass00},  for small time step $\kappa$, there exists a positive constant $C$ such that
\begin{equation}
   \sup_{t\geq 0} W_1(\rho_t^{(1)}, \tilde{\rho}_t^{(1)}) \leq C\kappa^{\frac{1}{2}}.
\end{equation}
\end{theorem}

\begin{proof}
We still consider the synchronous coupling (same L\'evy noise): for $1 \leq i\leq N$,
\begin{equation}
    X^{i}(t) = X^{i}(0) - \int_0^t \nabla V(X^i(s)) ds + \int_0^t \frac{1}{N-1} \sum_{j\neq i}K(X^j(s) - X^i(s)) ds +  L^{i}(t),
\end{equation}

\begin{equation}
    \tilde{X}^i(t) = \tilde{X}^i(0) - \int_{0}^t \nabla V(\tilde{X}^i(s)) ds + \int_0^t \frac{1}{p-1}\sum_{j\neq i, j \in \mathcal{C}_{m,q(i)}}K(\tilde{X}^j(s) - \tilde{X}^i(s)) ds +   L^i(t),
\end{equation}
and denote
\begin{equation}
    Z^i(t) := \tilde{X}^i(t) - X^i(t).
\end{equation}
For $\epsilon >0$, define
\begin{equation}
    |x|_\epsilon := \sqrt{|x|^2 + \epsilon^2}.
\end{equation}
Then we have
\begin{equation}
\begin{aligned}
|Z^1(t)|_\epsilon &= |Z^1(0)|_{\epsilon} + \int_0^t |Z^1(s)|_{\epsilon}^{-1} Z^1(s) \cdot \Big(-\big(\nabla V(\tilde{X}^1(s)) - \nabla V(X^1(s))\big)\\
&\quad + \big(\frac{1}{p-1}\sum_{j \neq 1, j \in \mathcal{C}_{m,q(1)}} K(\tilde{X}^j(s) - \tilde{X}^1(s)) - \frac{1}{N-1} \sum_{j\neq 1} K(X^j(s) - X^1(s))\big) \Big) ds.
\end{aligned}
\end{equation}
Taking expectation, since $X^1(0) = \tilde{X}^1(0)$, and using the strongly-convexity of $V(\cdot)$, the boundedness of $K(\cdot)$, as well as the exchangeability property of the particle system, we have
\begin{equation}
    \mathbb{E}|Z^1(t)|_\epsilon \leq \epsilon + \int_0^t \mathbb{E}\left[-(\lambda_V - 2L_K) |Z^1(s)|^{-1}_\epsilon |Z^1(s)|^2 \right] ds + \int_0^t \mathbb{E}\left[|Z^1(s)|^{-1}_\epsilon Z^1(s) \cdot \mathcal{X}_{m,1}(\tilde{X}(s)) \right]
\end{equation}
where
\begin{equation}
    \mathcal{X}_{m,1}(\tilde{X}(s)) := \frac{1}{p-1}\sum_{j\neq 1, j \in \mathcal{C}_{m,q(1)}} K(\tilde{X}^j(s) - \tilde{X}^1(s)) - \frac{1}{N-1} \sum_{j\neq 1}K(\tilde{X}^j(s) - \tilde{X}^1(s)).
\end{equation}
Clearly, from the definition of $|\cdot|_\epsilon$, we know that $|x|_\epsilon \leq |x| + \epsilon$. Then
\begin{multline}
    -\mathbb{E}\left[|Z^1(s)|^{-1}_\epsilon |Z^1(s)|^2\right] \leq  -\mathbb{E}\left[|Z^1(s)|^{-1}_\epsilon (|Z^1(s)|_\epsilon - \epsilon)^2\right]\\
    = \mathbb{E}\left[-|Z^1(s)|_\epsilon + 2\epsilon - \frac{\epsilon^2}{\sqrt{|Z^1(s)|^2 + \epsilon^2}}\right] \leq -\mathbb{E}|Z^1(s)|_\epsilon + 2\epsilon.
\end{multline}
For the remaining term, we denote it by
\begin{equation}
    \tilde{R}(t) := \mathbb{E}\left[|Z^1(t)|^{-1}_\epsilon Z^1(t) \cdot \mathcal{X}_{m,1}(\tilde{X}(t)) \right].
\end{equation}
Then
\begin{equation}
\begin{aligned}
    \tilde{R}(t) &= \mathbb{E}\left[|Z^1(t_m)|^{-1}_\epsilon Z^1(t_m) \cdot \mathcal{X}_{m,1}(\tilde{X}_{t_m})\right]\\
    &\quad + \mathbb{E}\left[|Z^1(t_m)|^{-1}_\epsilon Z^1(t_m) \cdot (\mathcal{X}_{m,1} (\tilde{X}(t)) - \mathcal{X}_{m,1}(\tilde{X}(t_m)))\right]\\
    &\quad+\mathbb{E}\left[\big(|Z^1(t)|^{-1}_\epsilon Z^1(t) - |Z^1(t_m)|^{-1}_\epsilon Z^1(t_m)\big) \cdot \mathcal{X}_{m,1}(\tilde{X}(t))\right]\\
    &:=J_1 + J_2 + J_3.
\end{aligned}
\end{equation}

Similarly with \eqref{eq:I10}, by consistency of the random batch,
\begin{equation}
    J_1 = 0.
\end{equation}

For the term $J_2$, since $|Z^1(t_m)|^{-1}_\epsilon |Z^1(t_m)| \leq 1$, by Lemma \ref{lmm:mathcalX}, we have
\begin{equation}
    J_2 \leq \mathbb{E} |\mathcal{X}_{m,1} (\tilde{X}(t)) - \mathcal{X}_{m,1}(\tilde{X}(t_m)) | \leq C\kappa.
\end{equation}

For the term $J_3$, we consider the function $f(x) := \frac{x}{|x|_\epsilon}$. Clearly, $\nabla f(x) = |x|_{\epsilon}^{-1}(I_d - \frac{x^{\otimes 2}}{|x|_{\epsilon}^2})$ is always well-defined, and $|\nabla f(x)| \leq 2/\epsilon$. Then, using the boundedness of $K(\cdot)$ (so that $\mathcal{X}_{m,1}(\tilde{X}(t))$ is bounded by $2M_K$), as well as the H\"older's continuity property of $Z^1$ in Lemma \ref{lmm:Zcon}, we have
\begin{equation}
    J_3 \leq \frac{2}{\epsilon}\mathbb{E}|Z^1(t) - Z^1(t_m)| \cdot 2M_K \leq C\frac{\kappa}{\epsilon}.
\end{equation}

Combining all the above, we conclude that
\begin{equation}
    \mathbb{E}|Z^1(t)|_\epsilon \leq \epsilon + \int_0^t (-C\mathbb{E}|Z^1(s)|_\epsilon + C(\epsilon + \kappa + \kappa \epsilon^{-1})) ds.
\end{equation}
Finally, choosing $\epsilon = \kappa^{\frac{1}{2}}$ for $\eta < 1$, by Gr\"onwall's inequality, and since $|x| < |x|_\epsilon$, we have 
\begin{equation}
    \mathbb{E}|Z^1(t)| \leq \mathbb{E}|Z^1(t)|_\epsilon \leq C\kappa^{\frac{1}{2}}.
\end{equation}
Then, by definition of the Wasserstein-1 distance, we obtain the desired uniform-in-time error estimate:
\begin{equation}
   \sup_{t \geq 0} W_1(\rho_t^{(1)}, \tilde{\rho}_t^{(1)}) \leq C\kappa^{\frac{1}{2}}.
\end{equation}
Above, the values of the constant $C$ may vary but are all independent of the time $t$ and the particle number $N$.

\end{proof}

A significant corollary of our result is the following $\alpha$-stable case.
\begin{assumption}\label{ass0}
The L\'evy measure $\nu$ satisfies
\begin{equation}
    \nu(dz) = C_{d,\alpha} |z|^{-(d+\alpha)}dz,\quad C_{d,\alpha} := \frac{2^{\alpha-1} \alpha \Gamma((d+\alpha) / 2)}{\pi^{2 / d} \Gamma(1-\alpha / 2)},\quad \alpha \in (1,2).
\end{equation}
\end{assumption}

\begin{corollary}[rotational invariant $\alpha$-stable case]
Let $\rho_t^{(1)}$, $\tilde{\rho}_t^{(1)}$ the first marginal distributions of \eqref{eq:ips}, \eqref{eq:rbm}, respectively. 
 Under Assumptions \ref{ass2}, \ref{ass:init1}, \ref{ass0},  for small time step $\kappa$, there exists a positive constant $C$ such that
\begin{equation}
   \sup_{t\geq 0} W_1(\rho_t^{(1)}, \tilde{\rho}_t^{(1)}) \leq C\kappa^{\frac{1}{2}}.
\end{equation}
\end{corollary}

Furthermore, a very natural question is: why we consider different measures (Wasserstein-2 distance in Theorem \ref{thm:secondmoment}) and Wasserstein-1 distance for Theorem \ref{thm:nosecondmoment})? To our knowledge, the choice of order of the Wasserstein distances is strongly related with the value of 
\begin{equation}
    \bar{p} := \sup\{p>0 : \mathbb{E}|L(t)|^p <\infty \},
\end{equation}
which is determined by the integrability of the L\'evy measure $\nu$. In other words, the convergence topology is determined by the intensity of large jumps. Moreover, after exactly the same derivation, we can see that if we replace 
\begin{equation}
    \int_{|z|\geq 1} |z|\nu(dz) < \infty
\end{equation}
by 
\begin{equation}
    \int_{|z|\geq 1} |z|^\alpha\nu(dz) < \infty
\end{equation}
for some $\alpha \in (1,2)$, then the result in Theorem \ref{thm:nosecondmoment} can be correspondingly replaced by 
\begin{equation}
   \sup_{t\geq 0} W_\alpha(\rho_t^{(1)}, \tilde{\rho}_t^{(1)}) \leq C\kappa^{\frac{1}{2}}.
\end{equation}
In this paper, we consider the $W_1$ distance just for simplicity, and we separately state the finite and infinite second moment cases because of the difference in the proof. We summarize the $W_\alpha$ ($\alpha \in (1,2)$) results discussed above into the following corollary:
\begin{assumption}[finite $p$-th moment]\label{asspp}
Fix $\alpha \in (1,2)$. The L\'evy measure $\nu$ satisfies
\begin{equation}
    \int_{|z|\geq 1} |z|^\alpha \nu(dz) < \infty.
\end{equation}
\end{assumption}

\begin{corollary}\label{thm:coroWp}
Let $\rho_t^{(1)}$, $\tilde{\rho}_t^{(1)}$ the first marginal distributions of \eqref{eq:ips}, \eqref{eq:rbm}, respectively. 
 Under Assumptions \ref{ass2}, \ref{ass:init1}, \ref{asspp},  for small time step $\kappa$, there exists a positive constant $C$ such that
\begin{equation}
   \sup_{t\geq 0} W_\alpha(\rho_t^{(1)}, \tilde{\rho}_t^{(1)}) \leq C\kappa^{\frac{1}{2}}.
\end{equation}
\end{corollary}

\section{Lemmas used in the proof}\label{sec:lemma}
In this section, we provide some technical lemmas used in the proof of Theorem \ref{thm:secondmoment}, \ref{thm:nosecondmoment}. The first lemma is the uniform-in-time moment bound for both the interacting particle system and the associated RBM dynamics. We consider the finite and infinite second moment cases separately.

\begin{lemma}[uniform-in-time moment estimate]\label{lmm:2ndmoment}
Suppose Assumptions \ref{ass2}, \ref{ass:init1} hold.
\begin{enumerate}
    \item Suppose Assumption \ref{ass1}  holds. Then there exists positive constant $C$ independent of $i$ such that
\begin{equation}
    \sup_{t \geq 0} \mathbb{E}|X^i(t)|^2 \leq C < \infty,\quad \sup_{t \geq 0} \mathbb{E}|\tilde{X}^i(t)|^2 \leq C < \infty,\quad 1\leq i\leq N.
\end{equation}
    \item Suppose Assumption \ref{ass0} holds. Then there exists positive constant $C$ independent of $i$ such that
\begin{equation}
    \sup_{t \geq 0} \mathbb{E}|X^i(t)| \leq C < \infty,\quad \sup_{t \geq 0} \mathbb{E}|\tilde{X}^i(t)| \leq C < \infty,\quad 1\leq i\leq N.
\end{equation}
\end{enumerate}

\end{lemma}

\begin{proof}
\begin{enumerate}[wide]
\item \textbf{L\'evy jump with finite second-order moments.} We first consider the process $X^i(t)$. By It\^o's formula, we have 
\begin{equation}
|X^i(t)|^2 = |X^i(0)|^2 + \int_0^t 2X^i(s-) \cdot dX^i(s) + \langle (X^i)^c\rangle_t + \sum_{s\leq t}\left(|X^i(s) - X^i(s-) - 2 X^{i}(s-) \Delta X^i(s)\right).
\end{equation}
Here, $\Delta X^i(s) := X^i(s) - X^i(s-)$, and $(X^i)^c$ denotes the continuous part of the semimartingale $X^i$. Recall that
\begin{equation}
    X^{i}(t) = X^{i}(0) - \int_0^t \nabla V(X^i(s)) ds + \int_0^t \frac{1}{N-1} \sum_{j\neq i}K(X^j(s) - X^i(s)) ds +  L^{i}(t),\quad 1 \leq i\leq N.
\end{equation}
Clearly, the continuous part of the semi-martingale $X^i$ has zero quadratic variation ($ \langle (X^i)^c\rangle_t = 0$), since there is no Brownian motion. Moreover, recall that
\begin{equation}
    L^i(t) =  \int_0^t \int_{|z|<1} z\tilde{N}^i(dx,dz) + \int_0^t \int_{|z| \geq 1} z N^i(ds, dz).
\end{equation}
Also note that $\int_0^t X^i(s-) d(X^i)^c_s = \int_0^t X^i(s) d(X^i)^c_s$, we have
\begin{equation}
\begin{aligned}
|X^i(t)|^2 &= |X^i(0)|^2 + \int_0^t 2X^i(s) \cdot \left(-\nabla V(X^i(s)) + \frac{1}{N-1}\sum_{j\neq i} K(X^j(s) - X^i(s)) \right) ds \\
&\quad + \int_0^t \int_{\mathbb{R}^d} (|X^i(s-) + z|^2 - |X^i(s-)|^2)\tilde{N}^i(dz,ds)\\
&\quad +  \int_0^t \int_{|z|\geq 1}(|X^i(s-) + z|^2 - |X^i(s-)|^2) \nu(dz) ds\\
&\quad +  \int_0^t \int_{|z|<1}(|X^i(s-) + z|^2 - |X^i(s-)|^2 - 2z\cdot X^i(s-))\nu(dz)ds.
\end{aligned}
\end{equation}
By Assumption \ref{ass2}, $V$ is $\lambda_V$-convex, and $K$ is $L_K$-Lipschitz. Then since $X^1,\dots, X^N$ are exchangeable, we have
\begin{multline}
    \mathbb{E}\int_0^t 2X^i(s) \cdot \left(-\nabla V(X^i(s)) + \frac{1}{N-1}\sum_{j\neq i} K(X^j(s) - X^i(s)) \right) ds\\
    \leq -2(\lambda_v -2 L_K)\int_0^t \mathbb{E}|X^i(s)|^2 ds.
\end{multline}
The term $\int_0^t \int_{\mathbb{R}^d} (|X^i(t-) + z|^2 - |X^i(t-)|^2)\tilde{N}^i(dz,dt)$ is a martingale, so it has zero mean. For the other terms involving the L\'evy meausre $\nu$, by Young's inequality, we have
\begin{equation}
\begin{aligned}
    &\quad \mathbb{E}\int_0^t \int_{|z|\geq 1}(|X^i(s-) + z|^2 - |X^i(t-)|^2) \nu(dz) ds\\
    &=  \int_0^t \int_{|z|\geq 1}|z|^2 \nu(dz) ds + 2\mathbb{E} \int_0^t \int_{|z|\geq 1}X^i(s-) \cdot z \nu(dz) ds\\
    &\leq  (1 + \frac{1}{2\delta}) (\int_{|z|\geq 1}|z|^2 \nu(dz)) t + 2\delta \nu(\{|z| \geq 1 \}) \int_0^t \mathbb{E}|X^i(s)|^2ds,
\end{aligned}
\end{equation}
where $\delta > 0$ is a constant to be determined. For the other term, we have
\begin{equation}
     \mathbb{E} \int_0^t \int_{|z|<1}(|X^i(s-) + z|^2 - |X^i(s-)|^2 - 2z\cdot X^i(s-))\nu(dz)ds =  (\int_{|z|<1}|z|^2\nu(dz))t.
\end{equation}
Finally, since $\nu$ is a Levy measure, $\int_{|z|<1}|z|^2\nu(dz) < \infty$. By Assumption \ref{ass1}, $\int_{|z|\geq 1}|z|^2 \nu(dz) < \infty$. Hence, concluding all the above, choosing small $\delta$ such that  $\delta \nu(\{|z| \geq 1 \})  <\lambda_v - 2L_K$, we have
\begin{equation}
    \mathbb{E}|X^i(t)|^2 \leq \mathbb{E}|X^i(0)|^2 - C_1\int_0^t \mathbb{E}|X^i(s)|^2 ds + C_2t.
\end{equation}
By Gr\"onwall's inequality, and since $\sup_i\mathbb{E}|X^i(0)|^2 < \infty$, we conclude that
\begin{equation}
    \sup_{t \geq 0} \mathbb{E}|X^i(t)|^2 \leq C < \infty.
\end{equation}

For $\tilde{X}$, using exactly the same argument, we can obtain
\begin{equation*}
    \mathbb{E}[|\tilde{X}^i(t)|^2|\mathcal{F}_{t_m}] \leq \mathbb{E}|\tilde{X}^i(0)|^2 - C_1\int_0^t \mathbb{E}[|\tilde{X}^i(s)|^2|\mathcal{F}_{t_m}]ds + C_2t,\quad t \in[t_m,t_{m+1}).
\end{equation*}\
Consequently,
\begin{equation}
    \sup_{t \geq 0} \mathbb{E}|\tilde{X}^i(t)|^2 \leq C < \infty.
\end{equation}

\item \textbf{L\'evy jump with infinite second-order moments.}

We first consider the process $X$. 
Recall that $|x|_\epsilon = \sqrt{|x|^2 + \epsilon^2}$ for $\epsilon > 0$. By It\^o's formula, we have
\begin{equation}\label{eq:afterito}
\begin{aligned}
|X^i(t)|_\epsilon &= |X^i(0)|_\epsilon + \int_0^t |X^i(s)|_{\epsilon}^{-1} X^i(s) \cdot (-\nabla V(X^i(s) + \frac{1}{N-1}\sum_{j \neq i}K(X^j(s) - X^i(s))) ds\\
&\quad + \int_0^t \int_{\mathbb{R}^d} (|X^i(s-) + z|_\epsilon - |X^i(s-)|_\epsilon)\tilde{N}^i(dz,ds)\\
&\quad +  \int_0^t \int_{|z|\geq 1}(|X^i(s-) + z|_\epsilon - |X^i(s-)|_\epsilon) \nu(dz) ds\\
&\quad +  \int_0^t \int_{|z|<1}(|X^i(s-) + z|_\epsilon - |X^i(s-)|_\epsilon - |X^i(s-)|_\epsilon X^i(s-) \cdot z)\nu(dz)ds. 
\end{aligned}
\end{equation}
Note that the second line above is a martingale which has zero expectation. For the first line, by Assumption \ref{ass2}, $V$ is $\lambda_V$-convex, and $K$ is $L_K$-Lipschitz. Then since $X^1,\dots, X^N$ are exchangeable, we have
\begin{multline}
    \mathbb{E}\int_0^t |X^i(s)|_{\epsilon}^{-1} X^i(s) \cdot \left(-\nabla V(X^i(s)) + \frac{1}{N-1}\sum_{j\neq i} K(X^j(s) - X^i(s)) \right) ds\\
    \leq \int_0^t (-\lambda_V\mathbb{E}\left[|X^i(s)|^{-1}_\epsilon|X^i(s)|^2\right] + M_K) ds \leq \int_0^t (-\lambda_V\mathbb{E}|X^i(s)|_\epsilon + 2\lambda_{V} \epsilon + M_K) ds,
\end{multline}
where in the last inequality we have used the equality $|x|_\epsilon \leq |x| + \epsilon$. For the third line in \eqref{eq:afterito}, since $|\nabla_x |x|_\epsilon| = |x|_{\epsilon}^{-1} |x| < 1$, we have
\begin{equation}
\mathbb{E}\int_0^t \int_{|z|\geq 1}(|X^i(s-) + z|_\epsilon - |X^i(s-)|_\epsilon) \nu(dz) ds \leq \int_0^t \int_{|z| \geq 1} |z| \nu(dz) ds.
\end{equation}
For the last line in \eqref{eq:afterito}, since $|\nabla^2_{xx} |x|_\epsilon| = |x|^{-1}_{\epsilon} |I_d - \frac{x^{\otimes 2}}{|x|^2_\epsilon}| \leq \frac{2}{\epsilon}$, we have
\begin{equation}\label{eq:afterafterito}
\mathbb{E} \int_0^t \int_{|z|<1}(|X^i(s-) + z|_\epsilon - |X^i(s-)|_\epsilon - |X^i(s-)|_\epsilon X^i(s-) \cdot z)\nu(dz)ds \leq \frac{2}{\epsilon} \int_0^t \int_{|z|<1}|z|^2\nu(dz)ds.
\end{equation}
Since $\nu$ is a L\'evy measure, we know that $\int_{|z|<1}|z|^2 \nu(dz) < \infty$. Also, by Assumption \ref{ass00}, $\int_{|z| \geq 1} |z| \nu(dz) < \infty$. Therefore, by \eqref{eq:afterito} - \eqref{eq:afterafterito}, there exists positive constant $C$ such that
\begin{equation}
\mathbb{E}|X^i(t)|_\epsilon \leq \mathbb{E}|X^i(0)|_\epsilon + \int_0^t (-\lambda_V \mathbb{E}|X^i(s)|_\epsilon + C(1 + \epsilon + \epsilon^{-1})) ds.
\end{equation}
By Gr\"onwall's inequality, there exists a positive constant $C'$ independent of $t$ such that
\begin{equation}
\mathbb{E}|X^i(t)|_\epsilon \leq C'(1 + \epsilon + \epsilon^{-1}).
\end{equation}
Finally, choosing $\epsilon = 1$, and since $|x| < |x|_\epsilon$, we have
\begin{equation}
    \sup_{t \geq 0} \mathbb{E}|X^i(t)| < \infty.
\end{equation}

For $\tilde{X}$, using exactly the same argument, for $\epsilon > 0$, we can obtain for $t \in[t_m,t_{m+1})$,
\begin{equation}
    \mathbb{E}[|\tilde{X}^i(t)|_\epsilon|\mathcal{F}_{t_m}] \leq \mathbb{E}|\tilde{X}^i(0)|^2 + \int_0^t \Big(-\lambda_V\mathbb{E}[|\tilde{X}^i(s)|_\epsilon|\mathcal{F}_{t_m}] + C(1 + \epsilon + \epsilon^{-1})\Big)ds.
\end{equation}\
Consequently, choosing $\epsilon = 1$,
\begin{equation}
    \sup_{t \geq 0} \mathbb{E}|\tilde{X}^i(t)|  < \infty.
\end{equation}

\end{enumerate}

\end{proof}

The following Lemma states the H\"older's continuity for the process $Z^1$ (recall $Z_i = \tilde{X}^i - X^i$). The derivation does not depend on the properties of the noise, so the proof is exactly the same with the Brownian motion case in \cite{jin2020random, cai2024convergence}.

\begin{lemma}[H\"older's continuity for the process $Z^1$]\label{lmm:Zcon}
Suppose Assumption \ref{ass2} holds.
\begin{enumerate}
\item  Under Assumptions \ref{ass00}, there exists a positive constant $C$ independent of $t$, $m$, $\kappa$ such that  
\begin{equation}
    \mathbb{E}|Z^1(t) - Z^1(t_m)|  \leq C\kappa,\quad t \in [t_m,t_{m+1}).
\end{equation}
\item Furthermore, under Assumptions \ref{ass1}, there exists a positive constant $C$ independent of $t$, $m$, $\kappa$ such that  
\end{enumerate}
\begin{equation}
    \mathbb{E}|Z^1(t) - Z^1(t_m)|^2  \leq C\kappa^2,\quad t \in [t_m,t_{m+1}).
\end{equation}
    
\end{lemma}

\begin{proof}
Since $K(\cdot)$ is bounded and $\nabla V(\cdot)$ is Lipschitz by Assumption \ref{ass2}, and using the uniform-in-time (first-order) moment bound obtained in Lemma \ref{lmm:2ndmoment}, we have
\begin{equation}
\begin{aligned}
    &\quad\mathbb{E}\left|Z^1(t) - Z^1(t_m) \right|
    =\mathbb{E}\Big|-\int_{t_m}^{t}  \left(\nabla V(\tilde{X}^1(s))-\nabla V(X^1(s))\right) ds\\
    &\quad+ \int_{t_m}^{t}  \left(\frac{1}{p-1}\sum_{j\in \mathcal{C}_k, j\neq 1} K(\tilde{X}^j(s) - \tilde{X}^1(s)) - \frac{1}{N-1}\sum_{j\neq 1} K(X^j(s) - X^1(s)) \right) ds \Big|\\
    &\leq L_V\int_{t_m}^{t}(\mathbb{E}|\tilde{X}^1(s)| + \mathbb{E}|X^1(s)|)ds + 2M_K (t - t_m) \leq C\kappa,
\end{aligned}
\end{equation}
for some positive $C$ independent of  $m$, $t$, $N$, $p$.

Similarly, using the uniform-in-time (first-order) moment bound obtained in Lemma \ref{lmm:2ndmoment}, as well as the H\"older's inequality, we have
\begin{equation}
\begin{aligned}
    &\quad\mathbb{E}\left|Z^1(t) - Z^1(t_m) \right|^2
    =\mathbb{E}\Big|-\int_{t_m}^{t}  \left(\nabla V(\tilde{X}^1(s))-\nabla V(X^1(s))\right) ds\\
    &\quad+ \int_{t_m}^{t}  \left(\frac{1}{p-1}\sum_{j\in \mathcal{C}_k, j\neq 1} K(\tilde{X}^j(s) - \tilde{X}^1(s)) - \frac{1}{N-1}\sum_{j\neq 1} K(X^j(s) - X^1(s)) \right) ds \Big|^2\\
    &\leq L_V(t-t_m)\int_{t_m}^{t}(\mathbb{E}|\tilde{X}^1(s)| + \mathbb{E}|X^1(s)|)ds + 2M_K (t - t_m)^2 \leq C\kappa^2, 
\end{aligned}
\end{equation}
for some positive $C$ independent of  $m$, $t$, $N$, $p$.

\end{proof}

The following lemma characterize the H\"older's continuity for the operator $\mathcal{X}_{m,1}$ (recall its definition in \eqref{eq:mathcalXdef}). As we can see below, the derivation relies on the fact that the first moment of the L\'evy measure is finite.

\begin{lemma}[H\"older's continuity for $\mathcal{X}_{m,1}$]\label{lmm:mathcalX}
Recall the definition of $\mathcal{X}_{m,1}$ in \eqref{eq:mathcalXdef}:
\begin{equation}
    \mathcal{X}_{m,1}(\tilde{X}(s)) := \frac{1}{p-1}\sum_{j\neq 1, j \in \mathcal{C}_{m,q(1)}} K(\tilde{X}^j(s) - \tilde{X}^1(s)) - \frac{1}{N-1} \sum_{j\neq 1}K(\tilde{X}^j(s) - \tilde{X}^1(s)),
\end{equation}
where $s \in [t_m,t_{m+1})$. Suppose Assumption \ref{ass2} holds. 
\begin{enumerate}
    \item Under Assumptions \ref{ass1}, there exists a positive constant $C$ independent of $t$, $m$, $\kappa$ such that  
\begin{equation}
\mathbb{E}\left[Z^1(t_m) \cdot (\mathcal{X}_{m,1} (\tilde{X}(t)) - \mathcal{X}_{m,1}(\tilde{X}(t_m)))\right] \leq C\kappa \left( \mathbb{E}|Z^1(t_m)|^2\right)^{\frac{1}{2}},\quad t \in [t_m,t_{m+1}),
\end{equation}
\item Under Assumptions \ref{ass00}, there exists a positive constant $C$ independent of $t$, $m$, $\kappa$ such that  
\begin{equation}
\mathbb{E}| \mathcal{X}_{m,1} (\tilde{X}(t)) - \mathcal{X}_{m,1}(\tilde{X}(t_m) | \leq C\kappa.
\end{equation}
\end{enumerate}

\end{lemma}

\begin{proof}
By definition,
\begin{multline}
    \mathbb{E}\left[\mathcal{X}_{m,1} \left( \tilde{X}(s)\right)-\mathcal{X}_{m,1}(\tilde{X}(t_m)) \mid \mathcal{F}_{t_m}\right]\\
    = \frac{1}{p-1}\sum_{j \in \mathcal{C}_{m,q(1)}},j\neq 1\mathbb{E}\left[\delta \tilde{K}^j(s)\mid \mathcal{F}_{t_m}\right] - \frac{1}{N-1}\sum_{j\neq 1}\mathbb{E}\left[\delta \tilde{K}^j(s)\mid \mathcal{F}_{t_m}\right],
\end{multline}
where
\begin{equation}\label{eq:deltaK}
\begin{aligned}
    \delta \tilde{K}^j(s) := K(\tilde{X}^j(s) - \tilde{X}^1(s)) - K(\tilde{X}^j(t_m) - \tilde{X}^1(t_m))
    \leq L_K\left|\delta \tilde{X}^1(s) - \delta \tilde{X}^j(s) \right|,
\end{aligned}
\end{equation}
and
\begin{equation}
    \delta \tilde{X}^j(s) := \tilde{X}^j(s) - \tilde{X}^j(t_m),\quad 1\leq j \leq N.
\end{equation}
It is remaining to estimate $\delta \tilde{X}$. Clearly,
\begin{multline}
    |\mathbb{E}\left[\delta \tilde{X}^i(s) \mid \mathcal{F}_{t_m}\right]| = |\int_{t_m}^s \mathbb{E}\left[-\nabla V(\tilde{X}^i(u))\mid\mathcal{F}_{t_m}\right] du\\
    + \int_{t_m}^s \mathbb{E}\left[\frac{1}{p-1}\sum_{j\neq i, j \in\mathcal{C}_{m,q(i)}} K(\tilde{X}^j(u) - K(\tilde{X}^i(u)))\mid\mathcal{F}_{t_m}\right] du + \int_{t_m}^s \int_{|z|\geq 1}  z\nu(dz) du + 0|
\end{multline}
Using boundedness of $K$, Lipschitz of $\nabla V$, finite first moment, and Lemma \ref{lmm:2ndmoment}, we have 
\begin{equation}
    |\mathbb{E}\left[\delta \tilde{X}^i(s) \mid \mathcal{F}_{t_m}\right]| \leq C(\kappa + |\int_{t_m}^t \mathbb{E}\left[\tilde{X}^i(u) \mid \mathcal{F}_{t_m}\right]|)
\end{equation}
Finally, taking expectation, using tower property of the conditional expectation and  H\"older's inequality, we have 
\begin{equation}
\begin{aligned}
&\quad\mathbb{E}\left[Z^1(t_m) \cdot (\mathcal{X}_{m,1} (\tilde{X}(t)) - \mathcal{X}_{m,1}(\tilde{X}(t_m)))\right]\\
&=\mathbb{E}\left[Z^1(t_m) \cdot \mathbb{E}\left[(\mathcal{X}_{m,1} (\tilde{X}(t)) - \mathcal{X}_{m,1}(\tilde{X}(t_m)))\mid \mathcal{F}_{t_m}\right]\right]\\
&\leq \left(\mathbb{E}|Z^t(t_m)|^2\right)^{\frac{1}{2}} \left( C\kappa^2 + C\mathbb{E}\left| \int_{t_m}^t \left(\frac{1}{N-1} \sum_{j\neq 1} + \frac{1}{p-1}\sum_{j\neq 1,j\in \mathcal{C}_{m,q(i)}}\right)\tilde{X}^j(u) du\right|^2\right)^{\frac{1}{2}}\\
&\leq\left(\mathbb{E}|Z^t(t_m)|^2\right)^{\frac{1}{2}} \left( C\kappa^2 + C\kappa \int_{t_m}^t\mathbb{E}\left|  \left(\frac{1}{N-1} \sum_{j\neq 1} + \frac{1}{p-1}\sum_{j\neq 1,j\in \mathcal{C}_{m,q(i)}}\right)\tilde{X}^j(u) \right|^2du\right)^{\frac{1}{2}}\\
&\leq \left(\mathbb{E}|Z^t(t_m)|^2\right)^{\frac{1}{2}} \left( C\kappa^2 + 2C\kappa \int_{t_m}^t  \left(\frac{1}{N-1} \sum_{j\neq 1} + \frac{1}{p-1}\sum_{j\neq 1,j\in \mathcal{C}_{m,q(i)}}\right)\mathbb{E}\left|\tilde{X}^j(u) \right|^2du\right)^{\frac{1}{2}}.
\end{aligned}
\end{equation}
By Assumption \ref{ass1}, using the (second-order) moment bound in Lemma \ref{lmm:2ndmoment}, we obtain 
\begin{equation}
\mathbb{E}\left[Z^1(t_m) \cdot (\mathcal{X}_{m,1} (\tilde{X}(t)) - \mathcal{X}_{m,1}(\tilde{X}(t_m)))\right] \leq C\kappa \left( \mathbb{E}|Z_{t_m}|^2\right)^{\frac{1}{2}}.
\end{equation}
The derivation of the other claim is similar. By tower property of the conditional expectation, we have
\begin{multline}
\mathbb{E}|\mathcal{X}_{m,1} (\tilde{X}(t)) - \mathcal{X}_{m,1}(\tilde{X}(t_m)| = \mathbb{E}\left[\mathbb{E}|\mathcal{X}_{m,1} (\tilde{X}(t)) - \mathcal{X}_{m,1}(\tilde{X}(t_m)| \mid \mathcal{F}_{t_m}\right]\\
\leq 4L_K \mathbb{E}\left[|\delta \tilde{X}^1(s)|\mid \mathcal{F}_{t_m}\right] \leq C\kappa,
\end{multline}
where we have used Assumption \ref{ass00} and the (first-order) moment bound in Lemma \ref{lmm:2ndmoment} in the last inequality.
\end{proof}

\section{Numerical examples}\label{sec:num}
In this section, we run some numerical examples to evaluate the RBM-L\'evy algorithm and verify our theory in Section \ref{sec:convergence}. The experiments were conducted on a server running Ubuntu 20.04 with a 5.15.0-88-generic kernel, dual Intel Xeon Silver 4216 CPUs (64 cores), 125 GB RAM, and four NVIDIA TITAN Xp GPUs with CUDA 12.2 support.

\subsection{A simple test example}
We first consider a simple artificial example to verify the results obtained in our theory. In detail, in what follows, we test the dependence of the errors on the time step $\kappa$, the particle number $N$, and the time $T$. In detail, we consider the following 1-dimensional interacting particle system with L\'evy jump:
\begin{equation}
    dX^i(t) = -a X^i(t) dt + \frac{1}{N-1}\sum_{j\neq i}\frac{X^j(t) - X^i(t)}{1 + |X^j(t) - X^i(t)|^2} dt + dL^i(t),
\end{equation}
where $a \geq 0$ is a positive constant, and $L(t)$ is a rotational invariant $\alpha$-stable L\'evy process with $\alpha =1.5$. The interaction is clearly smooth, bounded and with bounded derivatives. Moreover, it has a long-range interaction. In our experiments, we take the initial state from the following distribution using the Metropolis-Hastings MCMC algorithm:
\begin{equation}\label{eq:rho0}
    \rho_0(x) = \frac{\sqrt{4-x^2}}{2\pi}\textbf{1}_{\{|x| \leq 2 \}}.
\end{equation}
To evaluate the error numerically, we use
\begin{equation}
    \hat{E}_1(T) := \frac{1}{N}\sum_{i=1}^N |X^i(T) - \tilde{X}^i(T)|,
\end{equation}
Above, $X^i$ and $\tilde{X}^i$ are coupled with the same L\'evy jump and simulated with forward Euler's scheme with small time step $\tau = 2^{-15}$.

In Figure \ref{fig:rate}, we consider the error at time $T=1$.  We test the error $\hat{E}_1$ for RBM-L\'evy with $\kappa = 2^{-4}, 2^{-5}, 2^{-6}, 2^{-7}$. As we can see from Figure \ref{fig:rate}, for different $N$ and different $a$, the slope of the log-log curve is approximately $\frac{1}{2}$. This means the convergence rate is $O(\sqrt{\kappa})$, and is independent of $N$, which is consistent with our theoretical result.

\begin{figure}[htbp]
    \centering
    \begin{subfigure}{0.45\textwidth}
        \centering
        \includegraphics[width=\textwidth]{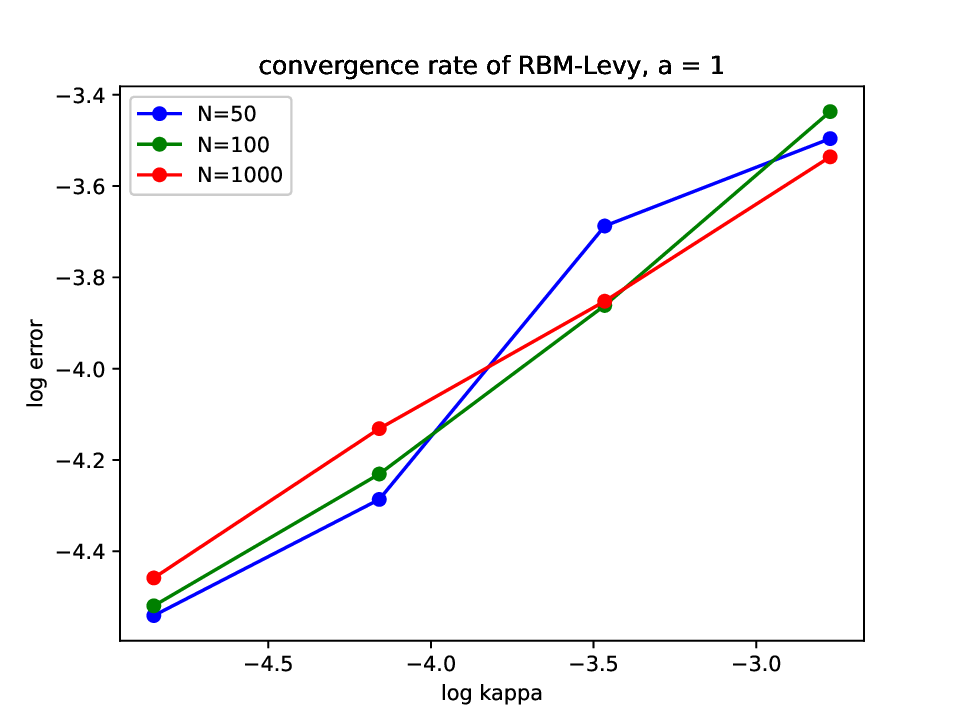}
        \caption{ }
        \label{fig:rbmlevy1d_rate}
    \end{subfigure}
    \hfill
    \begin{subfigure}{0.45\textwidth}
        \centering
        \includegraphics[width=\textwidth]{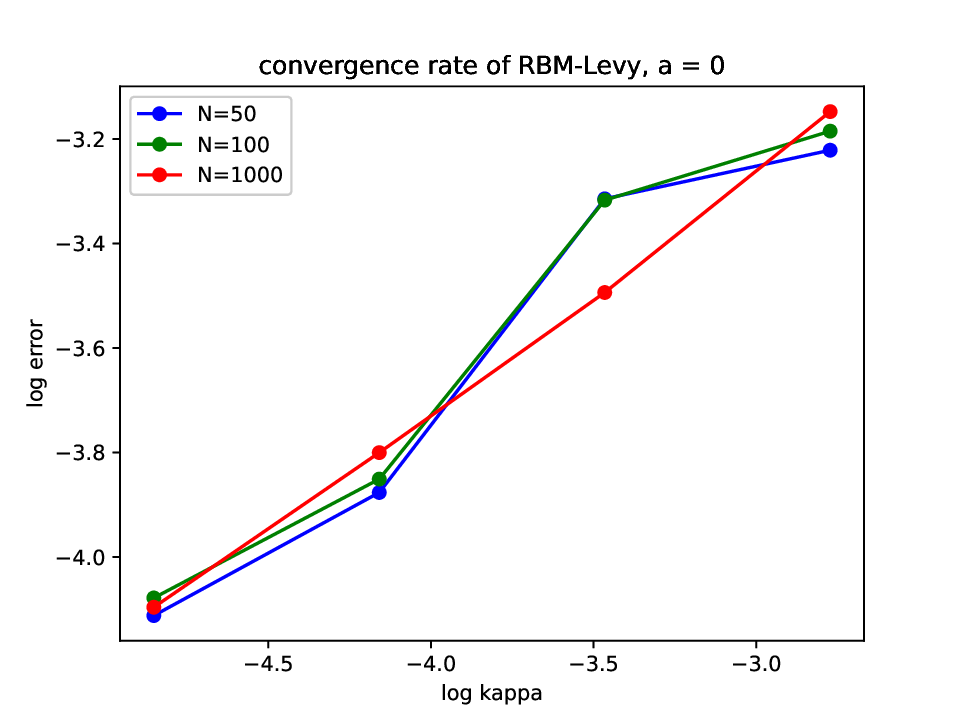}
        \caption{ }
        \label{fig:rvmlevy1d_rate_a0}
    \end{subfigure}
    \caption{Error of RBM-L\'evy versus time step at terminal time $T=1$ (log-log scale): \textbf{(a):} $a = 1$ (with confining condition), \textbf{(b):} $a=0$ (without confining condition). The particle number $N$ is $50, 100, 1000$ in each figure. }
    \label{fig:rate}
\end{figure}

In Figure \ref{fig:rbmlevy1d_time}, we test the long-time behavior of the RBM-L\'evy algorithm. We choose $\kappa = 2^{-7}$, $N = 100$, and $T = 1,2,4,8,16$. As we can see from Figure \ref{fig:rbmlevy1d_time}, when $a=1$ (namely, there is a strongly convex external potential), the convergence of RBM-L\'evy is uniform-in-time; when there is no such confining condition ($a=0$), the error of RBM would increase in time.

In figure \ref{fig:rvmlevy1d_cost}, we test the computational complexity of the RBM-L\'evy algorithm. We choose $a=1$, $\kappa = 2^{-7}$, $T=1$, and $N = 50,100,200,500,1000$. As we can see from the results, the computational cost of RBM-L\'evy is much cheaper.

\begin{figure}[htbp]
    \centering
    \begin{subfigure}{0.45\textwidth}
        \centering
        \includegraphics[width=\textwidth]{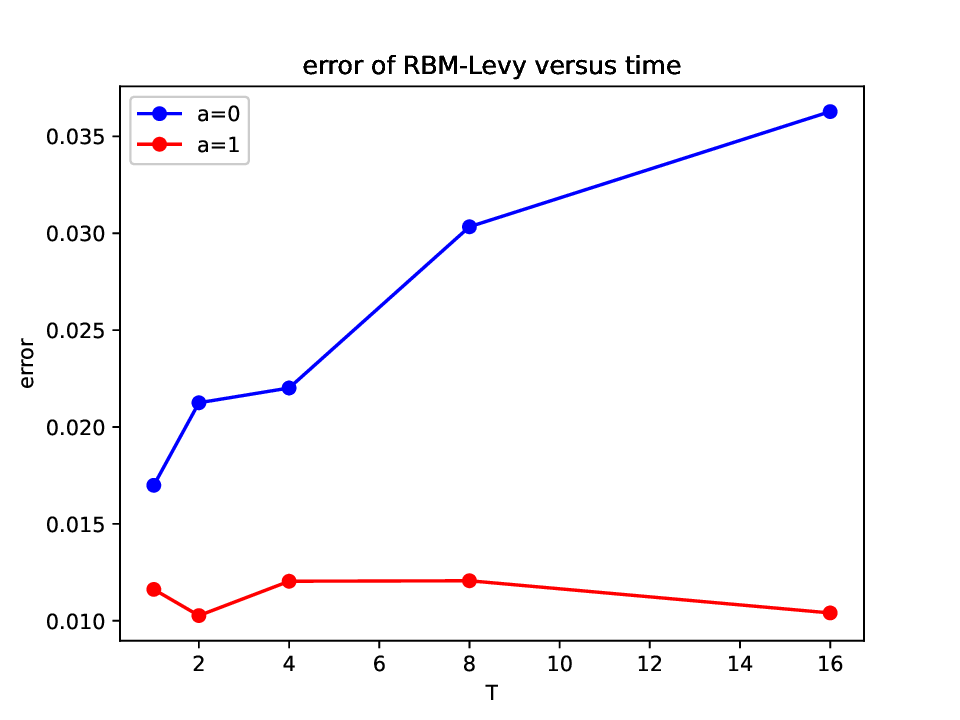}
        \caption{ }
        \label{fig:rbmlevy1d_time}
    \end{subfigure}
    \hfill
    \begin{subfigure}{0.45\textwidth}
        \centering
        \includegraphics[width=\textwidth]{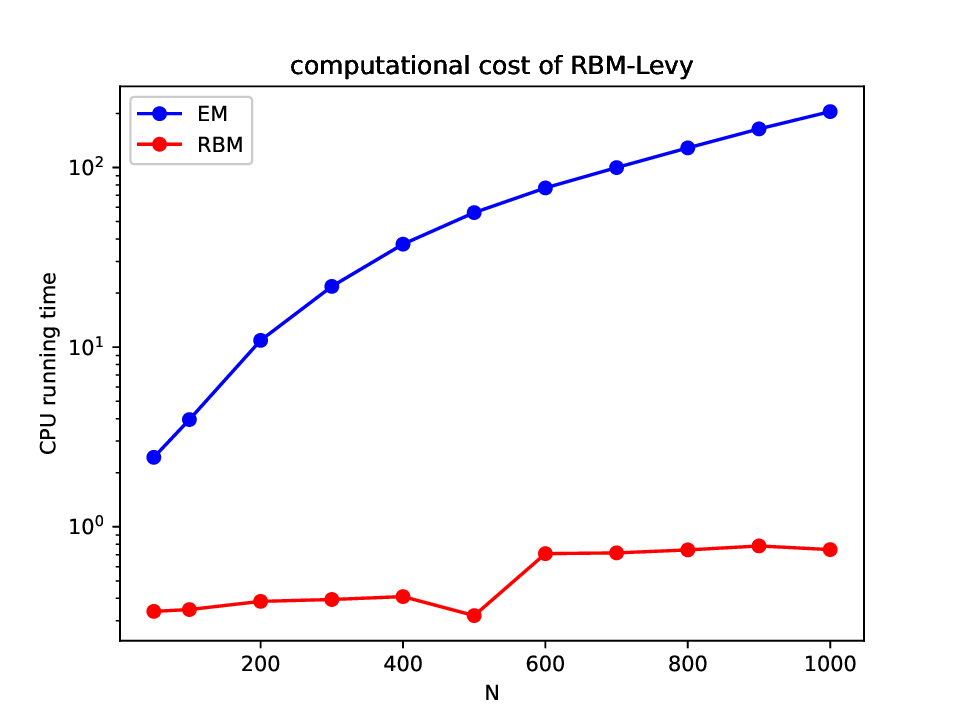}
        \caption{ }
        \label{fig:rvmlevy1d_cost}
    \end{subfigure}
    \caption{\textbf{(a):} Error of RBM-L\'evy versus time $T$ with configuration $T \in \{1,2,4,8,16 \}$. $\kappa = 2^{-7}$, $N=100$. \textbf{(b):} CPU running time versus particle number $N$. $a=1$, $\kappa = 2^{-7}$, $T=1$.}
    \label{fig:time_cost}
\end{figure}

\subsection{Application to stochastic Cucker-Smale model}\label{sec:sCS}
The Cucker-Smale model proposed by Cucker and Smale \cite{cucker2007emergent,cucker2007mathematics} applies a nonlinear second-order system to describe the behavior of particles swarm with Newton-type interaction. It can effectively describe the flocking behaviors of animals such as birds, fishes, and ants. The word ``flocking"  refers to general phenomena where autonomous agents reach a consensus based on limited environmental information and simple physical / social rules \cite{ha2009simple,ha2008particle,ha2021uniform}. Here we consider the following stochastic Cucker-Smale model, which can capture the random fluctuation in practice more precisely \cite{feng2024stochastic}:
\begin{equation}
\begin{aligned}
& dx_i(t) = v_i(t) dt,\\
& dv_i(t) = \frac{\theta}{N}\sum_{j=1}^N \Phi(|x_j(t) - x_i(t)|)(v_j(t) - v_i(t)) dt + (v_i(t-) - v_c(t))dL_i(t),
\end{aligned}
\end{equation}
where $X_i, v_i \in \mathbb{R}$, $\theta \in \mathbb{R}_{+}$, $v_c(t) := \frac{1}{N}\sum_{j=1}^N v_i(t)$,  $L_i(t) (1 \leq i \leq N)$ are independent L\'evy processes with characteristics $(0, \sigma^2, \nu(\cdot))$, and the interaction kernel is of the form:
\begin{equation}\label{eq:CSkernel}
    \Phi(r) := \frac{1}{(1 + r^2)^\beta}.
\end{equation}
It is known that when $\beta \in (0,\frac{1}{2}]$, the deterministic Cucker-Smale is unconditional flocking \cite{ha2009simple,ha2008particle,ha2021uniform}. Our experiment reveals that the Random Batch Method can effectively simulate the stochastic Cucker-Smale model whether the model is flocking or unflocking by choosing
\begin{equation}
    \beta = 5, \quad \theta = 1.
\end{equation}
and different settings of the noises. We also choose $\nu$ to be L\'evy measure of some compound Poisson process:  
\begin{equation}
    \nu(dz) = \lambda \cdot \frac{1}{\sqrt{2\pi}}\exp(-\frac{|z|^2}{2})dz
\end{equation}
for some $\lambda \geq 0$. Moreover, for each particle, the $x$-intitial distribution is i.i.d. chosen to be $\rho_0(x)$ (recall $\rho_0(x)$ in \eqref{eq:rho0}, and the $v$-initial distribution is i.i.d. chosen to be proportional to $\rho_0(0.1x)$.

In our experiments, we set the particle number $N=2^4$. The stochastic Cucker-Smale dynamics (with or without random batch on the drift) are simulated using Euler's scheme, and the time step for Euler scheme is $\tau = 2^{-8}$, the time step for choosing the random batch is $\kappa = 2^{-4}$, the batch size is $p=2$.

As we can see from Figure \ref{fig:BMjump} below, when $\sigma = 1$, $\lambda = 0.1$ (white noise + jump), the model is flocking and the Random Batch Method can effectively simulate the Cucker-Smale dynamics with much lower computational cost. Note that $D_x(t)$ and $D_v(t)$ are defined by
\begin{equation}
    D_x(t) := \max_{i,j}|x_i(t) - x_j(t)|,\quad D_v(t) := \max_{i,j}|v_i(t) - v_j(t)|.
\end{equation}
In the Appendix, we give some additional results under different settings of the noise, which reveals that the Random Batch Method can effectively simulate the model and preserve the flocking / unflocking nature of the system.

\begin{figure}[htbp]
\centering
        \begin{subfigure}[b]{0.4\textwidth}
            \includegraphics[width=\textwidth]{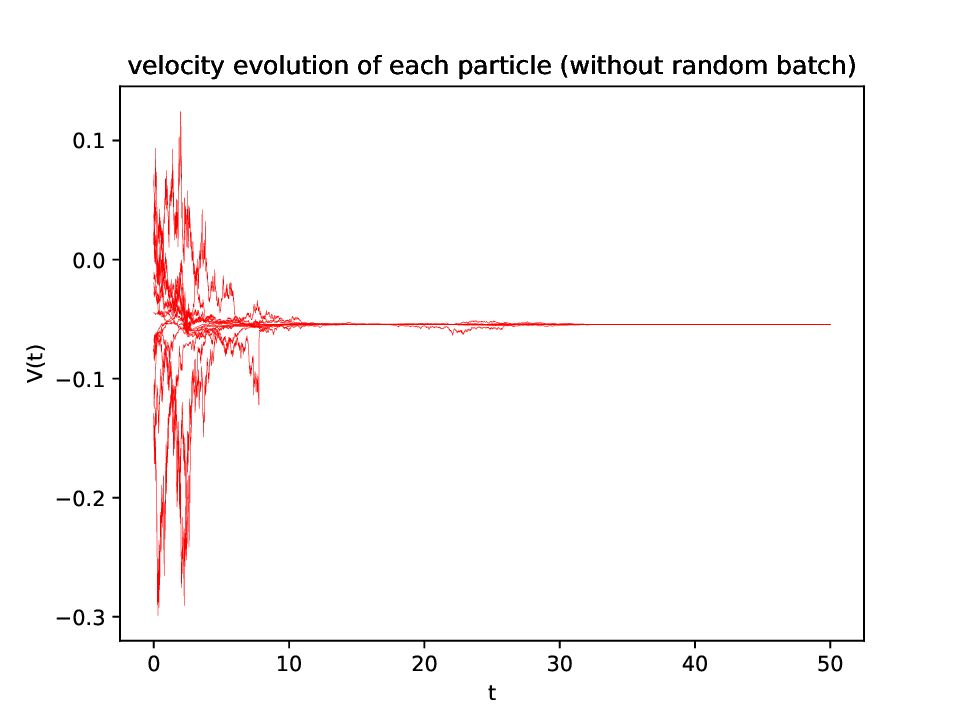}
            \caption{}
            \label{fig:image1}
        \end{subfigure}
        \quad
        \begin{subfigure}[b]{0.4\textwidth}
            \includegraphics[width=\textwidth]{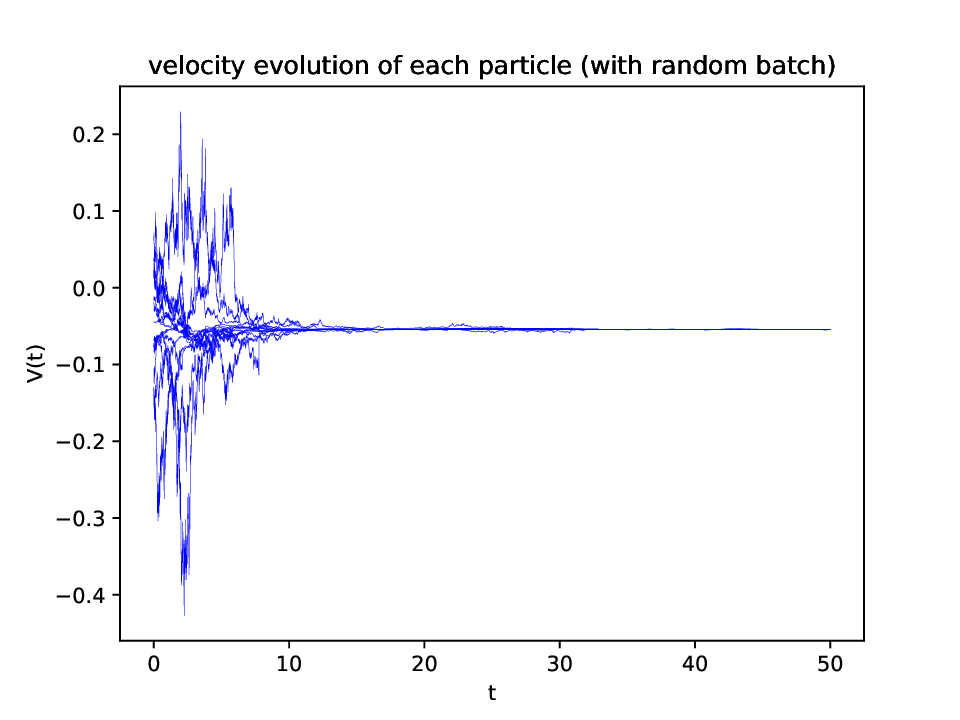}
            \caption{}
            \label{fig:image2}
        \end{subfigure}
\vskip\baselineskip
        \begin{subfigure}[b]{0.4\textwidth}
            \includegraphics[width=\textwidth]{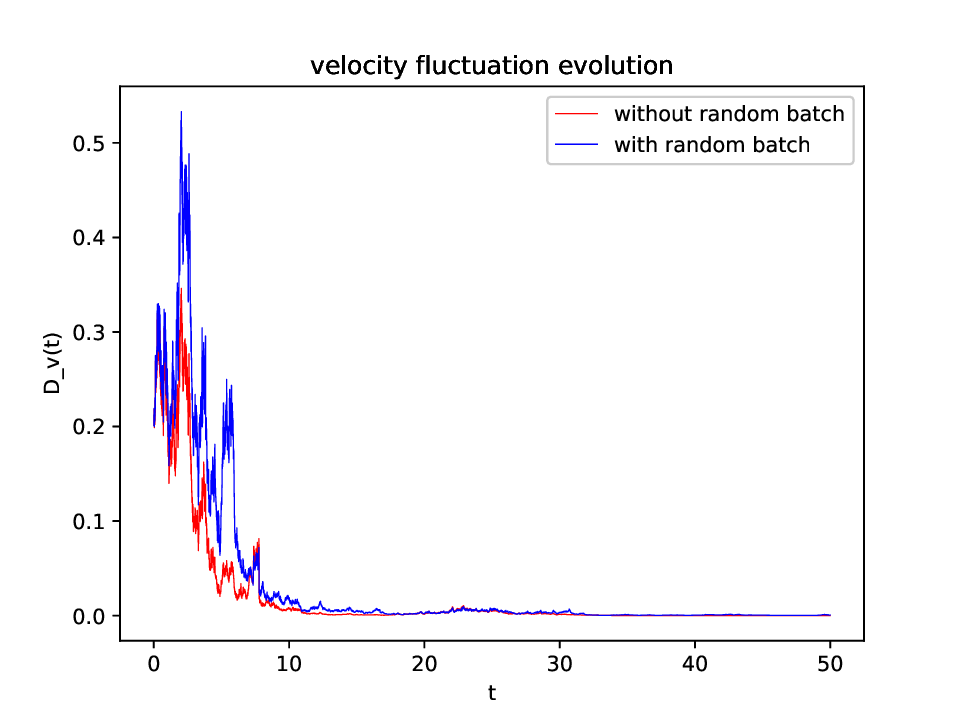}
            \caption{}
            \label{fig:image1}
        \end{subfigure}
        \quad
        \begin{subfigure}[b]{0.4\textwidth}
            \includegraphics[width=\textwidth]{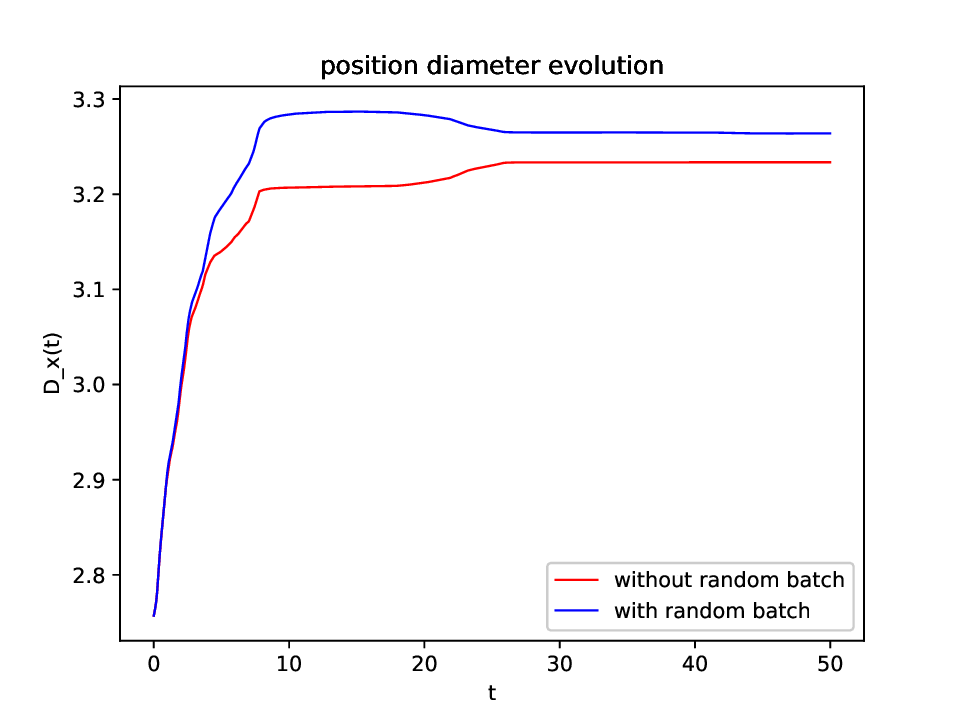}
            \caption{}
            \label{fig:image2}
        \end{subfigure}
        \caption{$\sigma = 1$, $\lambda = 0.1$ (white noise + jump). \textbf{flocking.} \textbf{(a):} velocity evolution of each particle without randam batch. \textbf{(b):} velocity evolution of each particle with random batch. \textbf{(c):} time evolution of $D_v(t)$ \textbf{(d):} time evolution of $D_x(t)$. 
        }
        \label{fig:BMjump}
\end{figure}

\section{Conclusion}
In this paper, we proposed the Random Batch Method for interacting particle systems driven by L\'evy noises (RBM-L\'evy), which can be viewed as an extension of the original RBM algorithm in \cite{jin2020random}. In our RBM-L\'evy algorithm, one randomly shuffles the $N$ particles into small batches of size $p$, and interacts particles only within each batch for a short period of time. Then one repeats this shuffle-and-interact procedure. Consequently, RBM-L\'evy dramatically reduces the computational cost from $O(N^2)$ to $O(pN)$ per time step, while still ensuring the convergence to the original interacting particle system, even when the noise allows jumps. Under the assumptions of both finite or infinite second moment of the L\'evy measure, we gave rigorous proof of this convergence in terms of Wasserstein distances. Typical applications include systems with $\alpha$-stable L\'evy noises for $\alpha \in (1,2)$. Remarkably, we are not including the theoretical results for $\alpha \in (0,1)$. This is mainly because our proof is based on the coupling technique, and it is known that an $\alpha$-stable process cannot have finite $\alpha'$-th moment for $\alpha'$ larger than $\alpha$. This then leads us to obtain a Wasserstein-$p$ convergence result for $p < \alpha$, but the Wasserstein-$p$ distance itself is not a true distance for $p < 1$. On the other hand, we do believe that it is possible to obtain a (sharp) Wasserstein convergence result using some other advanced PDE-based techniques instead of the current coupling method. We leave it as a nontrivial future work. Some numerical examples are given to verify our convergence rate and show the applicability of RBM-L\'evy. Some other possible future work may include extensions to particle systems with L\'evy noises and singular interaction kernels, and to particle systems with multiplicative noises with jumps.

\section*{Acknowledgement}
The research of J.-G. L. is partially supported under the NSF grant DMS-2106988. The authors would like to thank the editors and the anonymous reviewers for helpful comments and suggestions.

\appendix

\section{Some basics of L\'evy processes}\label{app:levy}

In this section, we refer to \cite{applebaum2009levy} for some basic definitions in probability theory associated with the L\'evy process. We refer to \cite{cohen2015stochastic} for more general jump processes. Let $L=(L(t),t\geq 0)$ be a stochastic process defined on a probability space $(\Omega, \mathcal{F}, P)$. We say that $L_t$ is a L\'evy process if
\begin{enumerate}
    \item $L(0) = 0$ a.s.;
    \item $L$ has independent and stationary increments;
    \item $L$ is continuous in probability, i.e. for all $a > 0$, $s \geq 0$,
    \begin{equation*}
        \lim_{t\rightarrow s}P\left(|L(t) - L(s)| > a \right) = 0.
    \end{equation*}
\end{enumerate}
Note that in the presence of the first two conditions, the third condition is equivalent to 
\begin{equation*}
        \lim_{t\rightarrow 0}P\left(|L(t)| > a \right) = 0.
\end{equation*}
Also, the second condition implies that for fixed $t \geq 0$, the random variable $L(t)$ is infinitely indivisible. This means for any $n \in \mathbb{N}_{+}$, there exists i.i.d. random variables $Y_1^{(n)},\dots, Y_n^{(n)}$ such that $L(t)$ equals to $\sum_{j=1}^n Y^{(n)}_j$ in the sense of distribution. Moreover, the characteristic function of $L(t)$ has the following L\'evy-Khintchine representation (here we consider the 1-dimension case L\'evy process for simplicity throughout this section):
\begin{equation*}
    \mathbb{E}\left[\exp(iu L(t))\right] = \exp(t\psi(u)),\quad \forall u \in\mathbb{R},
\end{equation*}
where the characteristic exponent of $L(t)$ is given by
\begin{equation}
    \psi(u) = i b u - \frac{1}{2} \sigma^2 u^2 + \int (e^{iu\cdot z} - 1 - iu z\textbf{1}_{\{ |z|\leq 1\}})\nu(dz).
\end{equation}
Here $b, \sigma \in \mathbb{R}$,  and $\nu$ is the L\'evy measure of $L$ satisfying
\begin{equation*}
    \nu(\{ 0\}) = 0,\quad \int 1 \wedge |z|^2 \nu(dz) < \infty.
\end{equation*}
Correspondingly, for any L\'evy process $L(t)$, we have the following It\^o-L\'evy decomposition:
\begin{equation}
    L(t) = b t + \sigma B_t + \int_0^t \int_{|z|<1} z\tilde{N}(ds,dz) + \int_0^t \int_{|z|\geq 1} zN(ds,dz).
\end{equation}
where $N(dt,dz)$ is the Poisson measure and $\tilde{N}(dt,dz):= N(dt,dz) - \mathbb{E}N(dt,dz) = N(dt,dz) - \nu(dz)dt$ is the compensated Poisson measure. Above, the triple $(b, \sigma^2, \nu)$ is often called the characteristic of a L\'evy process. Moreover, the jump process above as a semimartingale also satisfies the It\^o's formula. In detail, letting $X$ be a semimartingale and $f$ be a twice differentiable function, we have
\begin{equation}
    f(X_t) = f(X_0) + \int_0^t f'(X_{s-}) dX_s + \frac{1}{2}\int_0^tf''(X_s)d\langle X^c \rangle_s + \sum_{0 < s \leq t}(f(X_s) - f(X_{s}) - f'(X_{s-})\Delta X_s),
\end{equation}
where $X^c$ denotes the continuous part of the semimartingale $X$. In particular, if $X$ solves the following  SDE driven by a (pure jump) L\'evy process with characteristic $(0,0,\nu)$:
\begin{equation}
    X_t = X_0 + \int_0^t b(X_s) ds +  L(t), \quad t \geq 0,
\end{equation}
then 
\begin{equation}
\begin{aligned}
    f(X_t) &= f(X_0) + \int_0^t f'(X_s) b(X_s) ds +  \int_0^t \int_{\mathbb{R}} (f(X_{s-} + z) - f(X_{s-})) \tilde{N}(dz,ds)\\
    &\quad +  \int_0^t \int_{|z| \geq 1} (f(X_{s-} + z) - f(X_{s-})) \nu(dz) ds\\
    &\quad +  \int_0^t \int_{|z| < 1} (f(X_{s-} + z) - f(X_{s-}) - f'(X_{s-})z)\nu(dz) ds.
\end{aligned}
\end{equation}

\subsection{Rotational invariant $\alpha$-stable L\'evy processes}
A real-valued random variable $X$ is called stable if there exists real-valued sequences $(c_n)_n$, $(d_n)_n$ with each $c_n \geq 0$ such that
\begin{equation}
    X_1 + X_2 + \dots + X_n \overset{d}{=} c_n X + d_n,
\end{equation}
where $X_1, \dots, X_n$ are $n$ independent of $X$, and it can be proved that the only choice of the sequence $c_n$ is $C_n = \sigma n^{\frac{1}{\alpha}}$ for some $\sigma > 0$ and $\alpha \in (0,2]$ being its index of stability [\cite{feller1991introduction}, pp.166]. Alternatively, a stable process can be defined via its characteristic function: A real-valued random variable $X$ is stable if and only if there exist $\sigma > 0$ , $\beta \in [-1,1]$, $\alpha \in (0,2]$ and $b \in \mathbb{R}$ such that
the characteristic function of $X$ is given by
\begin{equation}\label{eq:stablecharacteristic}
\psi(u)=\mathbb{E}\left[e^{i u X}\right]= \begin{cases}\exp \left(i b u-\frac{1}{2} \sigma^2 u^2\right), & \text { if } \alpha=2, \\ \exp \left\{i b u-\sigma^\alpha|u|^\alpha\left[1-i \beta \operatorname{sgn}(u) \tan \left(\frac{\pi \alpha}{2}\right)\right]\right\}, & \text { if } \alpha \neq 1,2, \\ \exp \left\{i b u-\sigma|u|\left[1+i \beta \frac{2}{\pi} \operatorname{sgn}(u) \log (|u|)\right]\right\}, & \text { if } \alpha=1\end{cases}
\end{equation}
In particular, $X$ is called rotational invariant $\alpha$-stable if $\beta = b = 0$, namely, its L\'evy symbol $\eta(u) := \log(\psi(u))$ satisfies
\begin{equation}
    \eta(u) = -\tilde{\sigma}^\alpha |u|^\alpha,
\end{equation}
where $\tilde{\sigma} = \sigma$ for $\alpha \neq 2$ and $\tilde{\sigma} = \sigma / \sqrt{2}$ if $\alpha = 2$. It is also known that $X$ has finite $p$-th moment when $\alpha > p$.

Now, for $\alpha \in (0,2)$, a real-valued L\'evy process $L_\alpha(t)$ is called $\alpha$-stable if for any $t \geq 0$, the random variable $L(t)$ is stable. In particular, in this paper we focus on the rotational invariant $\alpha$-stable L\'evy process, where the L\'evy symbol of $L(1)$ is $\eta(u) = -|u|^\alpha$ (we take $\sigma = 1$ in \eqref{eq:stablecharacteristic} for simplicity). Consequently, the characteristics of $L_\alpha(t)$ is $(0,0,\nu_\alpha)$, where
\begin{equation}
    \nu(dz) = C_{\alpha} |z|^{-(1+\alpha)}dz,\quad C_{\alpha} := \frac{2^{\alpha-1} \alpha \Gamma((1+\alpha) / 2)}{\pi^{2} \Gamma(1-\alpha / 2)}.
\end{equation}
Furthermore, it is well-known that the rotational invariant $\alpha$-stable L\'evy process corresponds to the fractional Laplacian operator $(-\Delta)^{\frac{\alpha}{2}}$, which is a non-local one and can be defined through the Fourier transform (see for instance, [\cite{stein1970singular}, Chapter 5]):
\begin{equation}
    (-\Delta)^{\frac{\alpha}{2}} \rho (x) := \mathcal{F}^{-1}(|\xi|^\alpha \hat{\rho}(\xi))(x),\quad \hat{\rho}(\xi) = \mathcal{F}(\rho(x))(\xi).
\end{equation}
In detail, the law $\rho_t$ of $L_\alpha(t)$ solves a fractional heat equation:
\begin{equation}
    \partial_t \rho = - (-\Delta)^{\frac{\alpha}{2}} \rho.
\end{equation}
Clearly, this reduces to the classical heat equation when $\alpha = 2$, namely, $L_\alpha(t)$ is a Brownian motion.

\section{Additional results for the stochastic Cucker-Smale model}

In this section, we provide more results of the random batch simulation of the stochastic Cucker-Smale model discussed in Section \ref{sec:sCS}. In detail, we consider another three different settings (two of them are not flocking): (1) $\sigma = 1$, $\lambda = 0$ (white noise only, Figure \ref{fig:BM}); (2) $\sigma = 0$, $\lambda = 0.1$ (jump only, Figure \ref{fig:jump}); (3) $\sigma = 0$, $\lambda = 0$ (deterministic, Figure \ref{fig:det}). The particle number is $N=2^4$. The stochastic Cucker-Smale dynamics (with or without random batch on the drift) are simulated using Euler's scheme, and the time step for Euler scheme is $\tau = 2^{-8}$, the time step for choosing the random batch is $\kappa = 2^{-4}$, the batch size is $p=2$.  

As we can observe from Figures \ref{fig:BM} - \ref{fig:det}, the existence of Brownian motion somehow determines the flocking or unflocking behavior of the system. To our knowledge, this phenomenon has no theoretical explanations under the current choice of the interaction kernel in \eqref{eq:CSkernel}. The random batch approach, in the current settings, can effectively simulate the original model and capture the flocking / unflocking nature, with much cheaper computational cost. Furthermore, inspired by these observations, it is meaningful to explore the nature of the fact that continuous noise makes flocking easier, which is of independent interest.

\begin{figure}[htbp]
\centering
        \begin{subfigure}[b]{0.4\textwidth}
            \includegraphics[width=\textwidth]{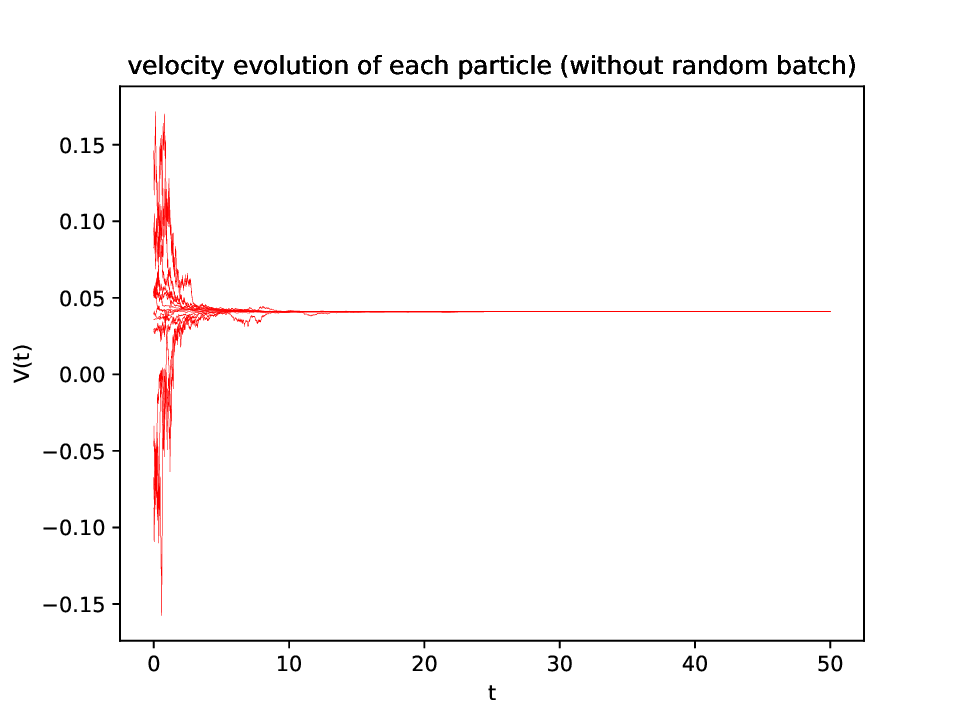}
            \caption{}
            \label{fig:image1}
        \end{subfigure}
        \quad
        \begin{subfigure}[b]{0.4\textwidth}
            \includegraphics[width=\textwidth]{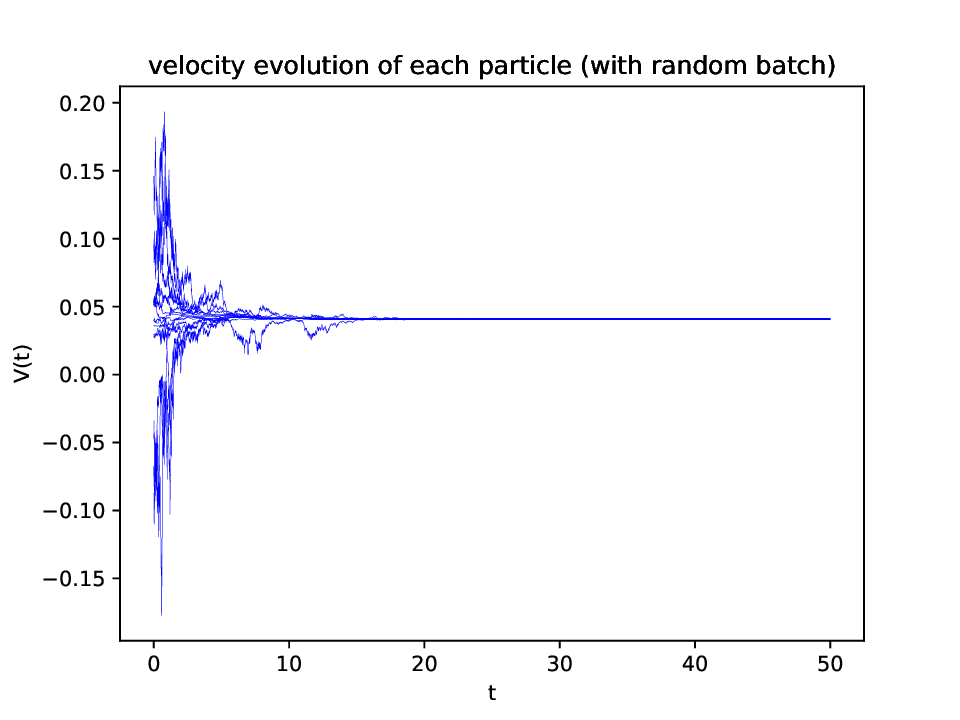}
            \caption{}
            \label{fig:image2}
        \end{subfigure}
\vskip\baselineskip
        \begin{subfigure}[b]{0.4\textwidth}
            \includegraphics[width=\textwidth]{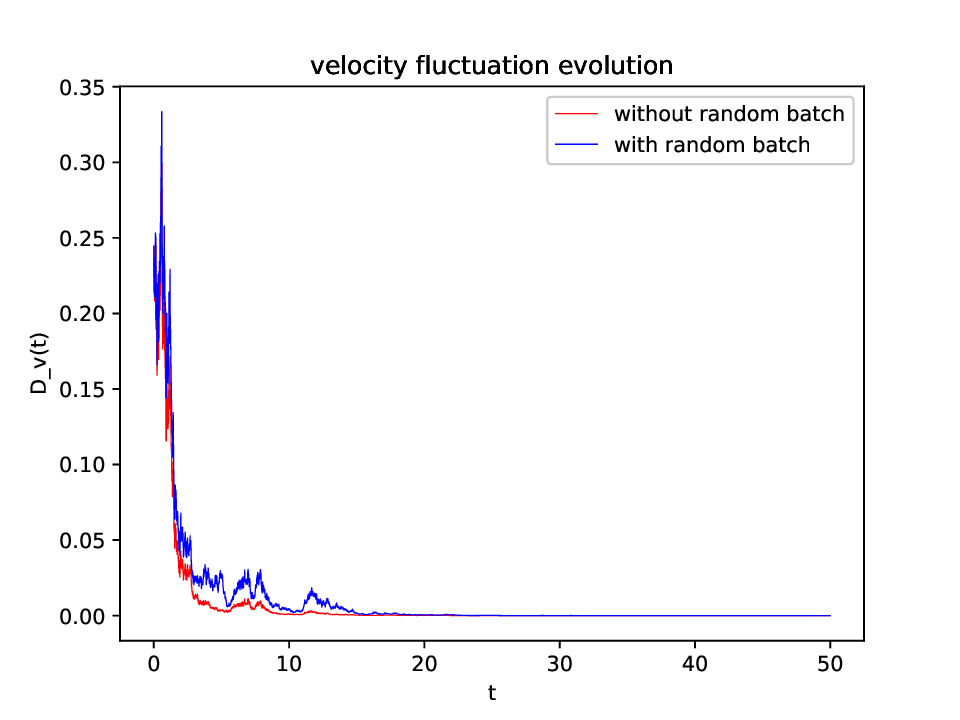}
            \caption{}
            \label{fig:image1}
        \end{subfigure}
        \quad
        \begin{subfigure}[b]{0.4\textwidth}
            \includegraphics[width=\textwidth]{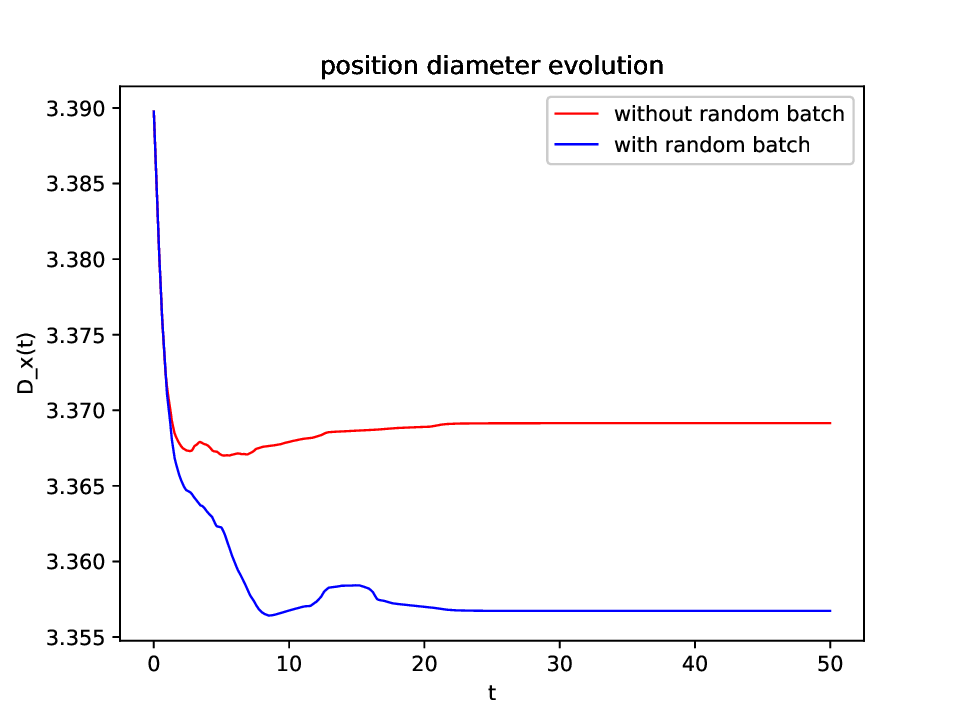}
            \caption{}
            \label{fig:image2}
        \end{subfigure}
        \caption{$\sigma = 1$, $\lambda = 0$ (white noise only). \textbf{flocking.} \textbf{(a):} velocity evolution of each particle without randam batch. \textbf{(b):} velocity evolution of each particle with random batch. \textbf{(c):} time evolution of $D_v(t)$ \textbf{(d):} time evolution of $D_x(t)$. }
        \label{fig:BM}
\end{figure}

\begin{figure}[htbp]
\centering
        \begin{subfigure}[b]{0.4\textwidth}
            \includegraphics[width=\textwidth]{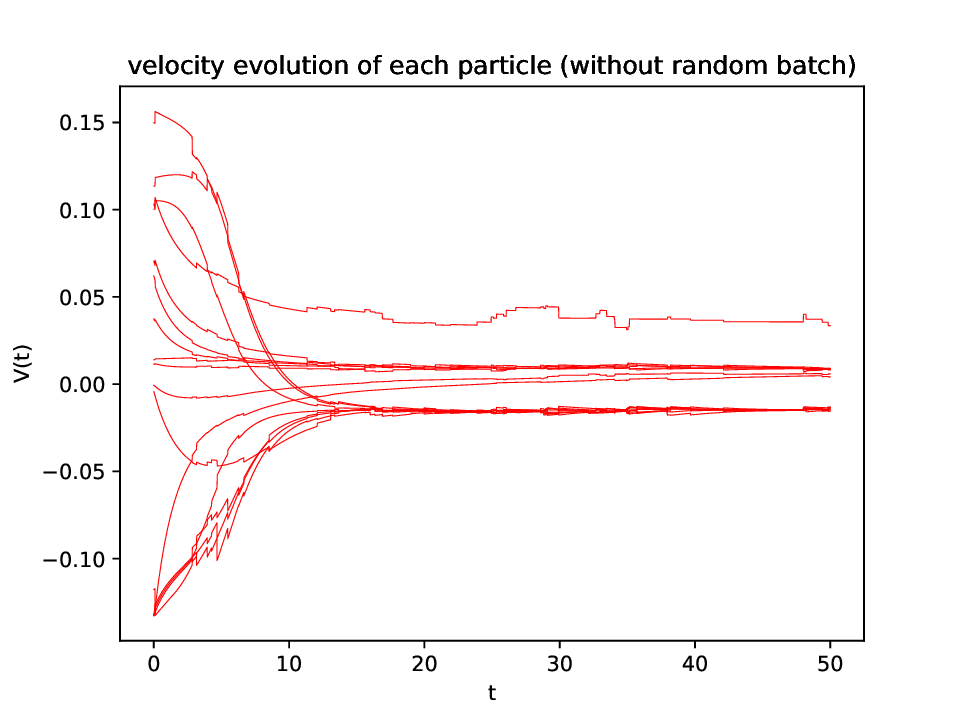}
            \caption{}
            \label{fig:image1}
        \end{subfigure}
        \quad
        \begin{subfigure}[b]{0.4\textwidth}
            \includegraphics[width=\textwidth]{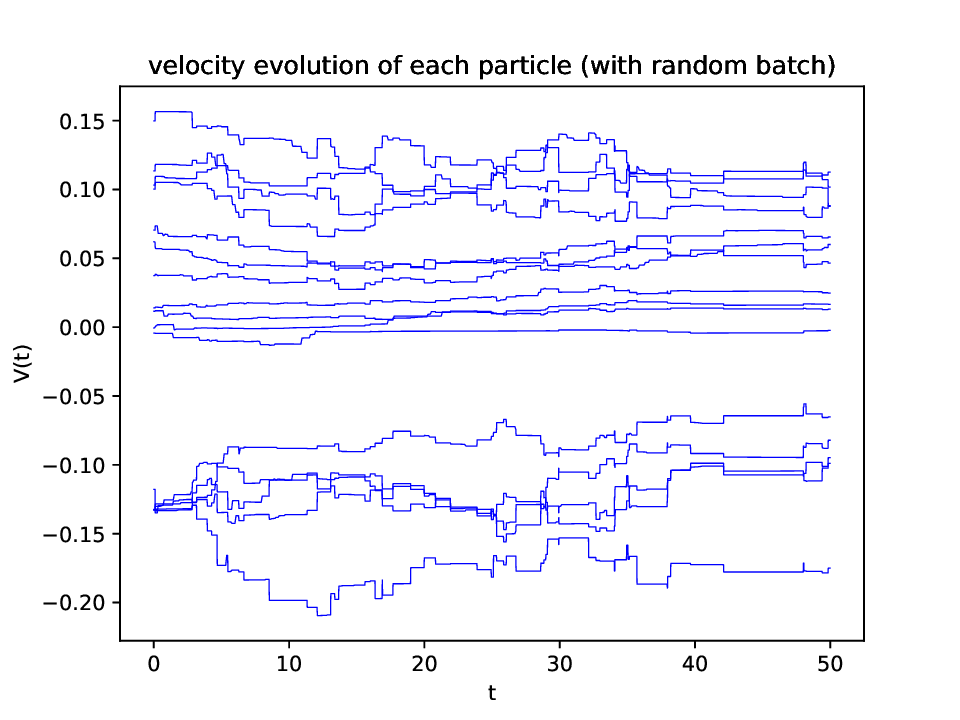}
            \caption{}
            \label{fig:image2}
        \end{subfigure}
\vskip\baselineskip
        \begin{subfigure}[b]{0.4\textwidth}
            \includegraphics[width=\textwidth]{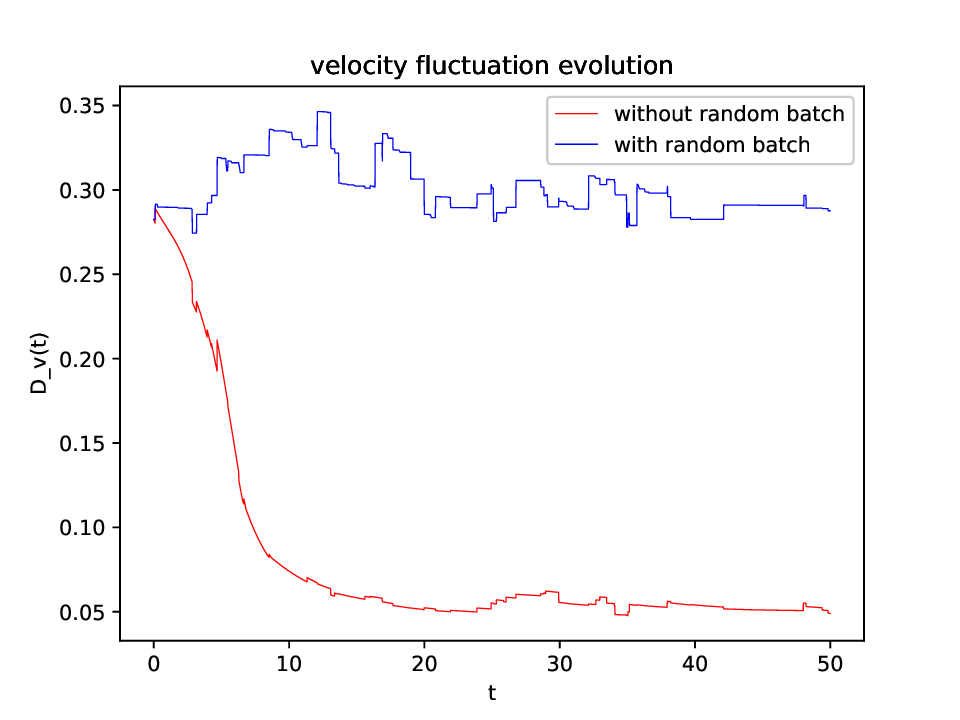}
            \caption{}
            \label{fig:image1}
        \end{subfigure}
        \quad
        \begin{subfigure}[b]{0.4\textwidth}
            \includegraphics[width=\textwidth]{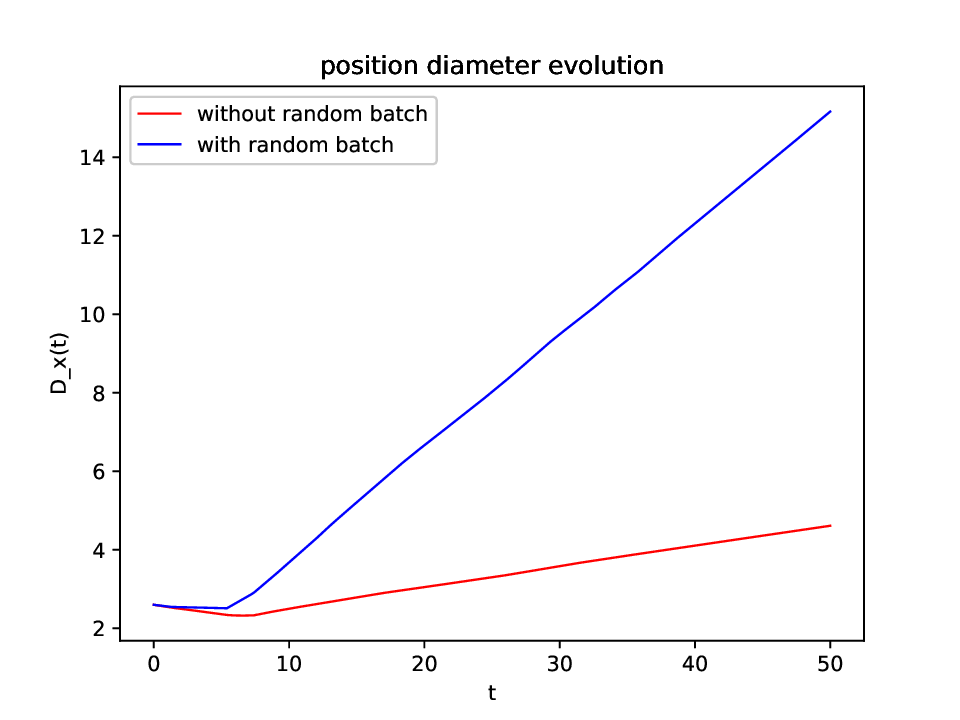}
            \caption{}
            \label{fig:image2}
        \end{subfigure}
        \caption{$\sigma = 0$, $\lambda = 0.1$ (jump only). \textbf{unflocking.} \textbf{(a):} velocity evolution of each particle without randam batch. \textbf{(b):} velocity evolution of each particle with random batch. \textbf{(c):} time evolution of $D_v(t)$ \textbf{(d):} time evolution of $D_x(t)$. }
        \label{fig:jump}
\end{figure}

\begin{figure}[htbp]
\centering
        \begin{subfigure}[b]{0.4\textwidth}
            \includegraphics[width=\textwidth]{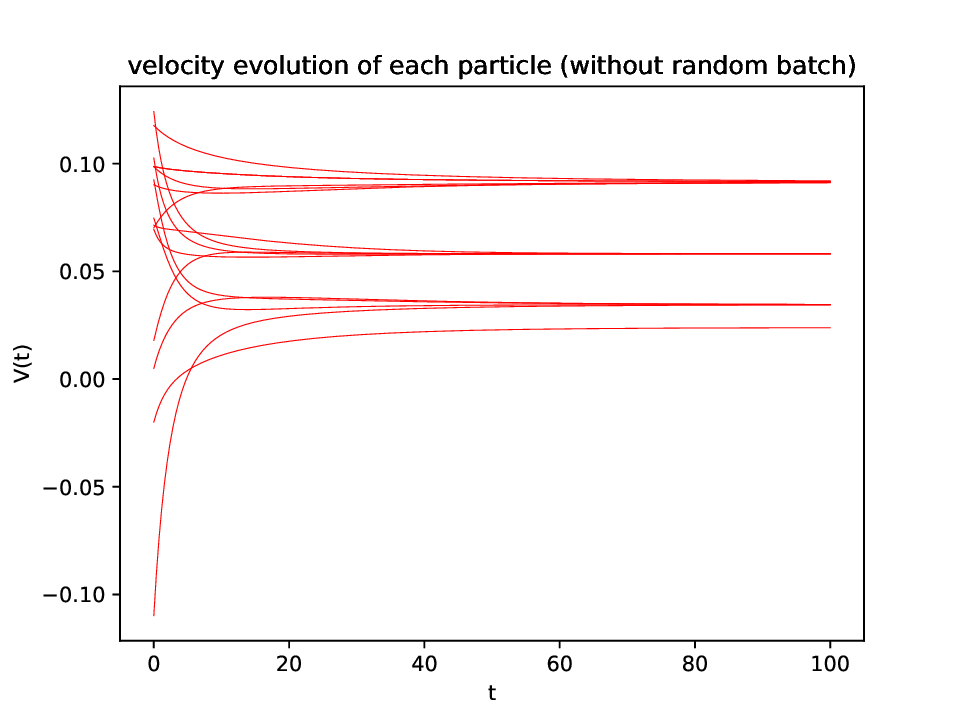}
            \caption{}
            \label{fig:image1}
        \end{subfigure}
        \quad
        \begin{subfigure}[b]{0.4\textwidth}
            \includegraphics[width=\textwidth]{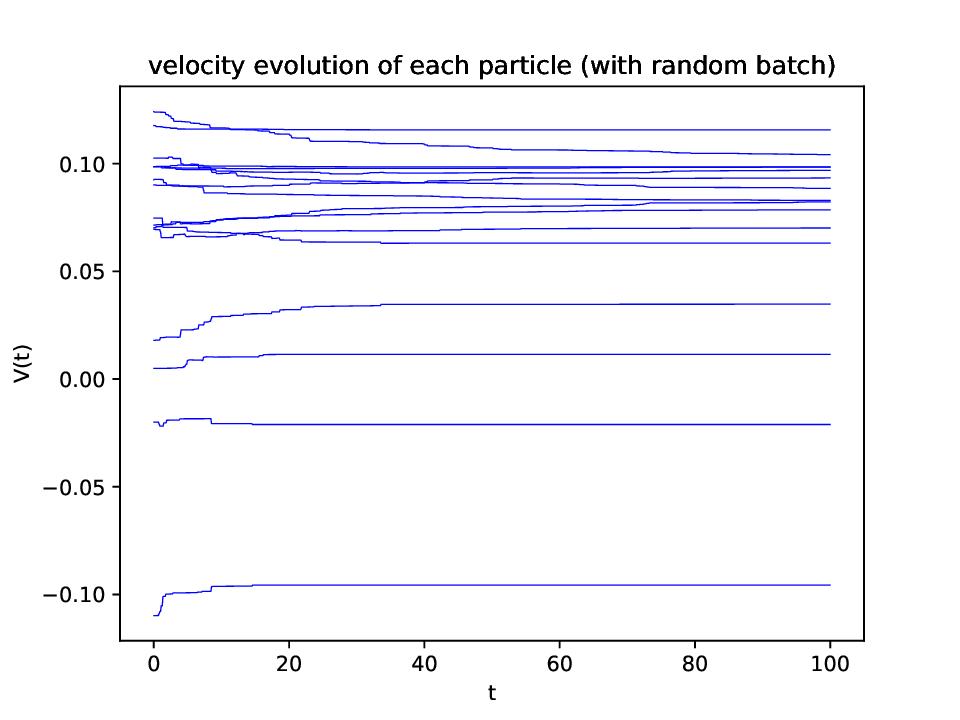}
            \caption{}
            \label{fig:image2}
        \end{subfigure}
\vskip\baselineskip
        \begin{subfigure}[b]{0.4\textwidth}
            \includegraphics[width=\textwidth]{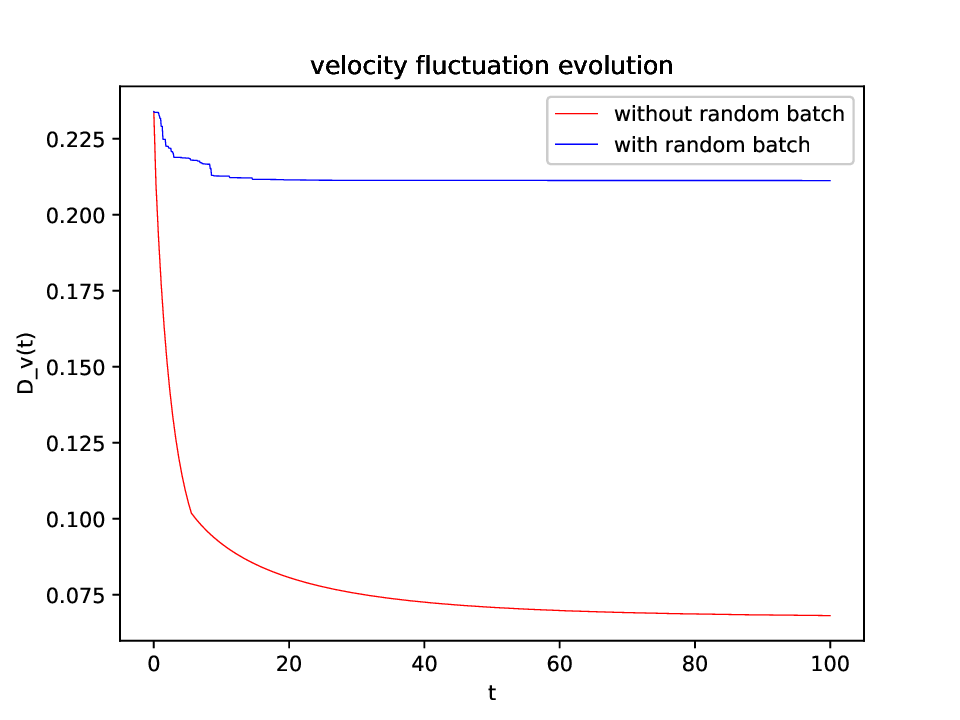}
            \caption{}
            \label{fig:image1}
        \end{subfigure}
        \quad
        \begin{subfigure}[b]{0.4\textwidth}
            \includegraphics[width=\textwidth]{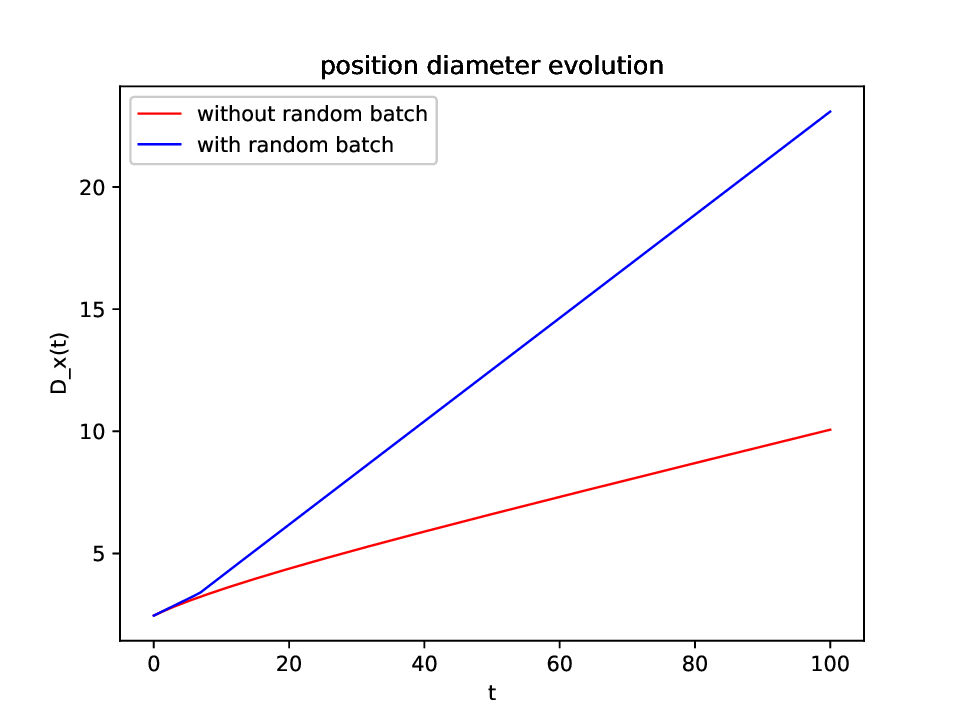}
            \caption{}
            \label{fig:image2}
        \end{subfigure}
        \caption{$\sigma = 0$, $\lambda = 0$ (deterministic). \textbf{unflocking.} \textbf{(a):} velocity evolution of each particle without randam batch. \textbf{(b):} velocity evolution of each particle with random batch. \textbf{(c):} time evolution of $D_v(t)$ \textbf{(d):} time evolution of $D_x(t)$. }
        \label{fig:det}
\end{figure}

\newpage
\bibliographystyle{plain}
\bibliography{main}

\end{document}